\documentclass[twoside,11pt]{article}

\usepackage{jmlr2e}

\jmlrheading{1}{2000}{1-48}{4/00}{10/00}{meila00a}{Marina Meil\u{a} and Michael I. Jordan}

\ShortHeadings{Sparse solutions of kernel herding}{Kazuma and Ken'ichiro}
\firstpageno{1}

\usepackage[T1]{fontenc}
\usepackage[utf8]{inputenc}

\usepackage{lmodern}
\usepackage{microtype}
\usepackage{csquotes}

\usepackage{mathtools}
\usepackage{color}
\usepackage{wrapfig}
\usepackage{float}
\usepackage{url}
\usepackage{boites}
\usepackage{mathrsfs}
\usepackage{algorithm}
\usepackage{algorithmic}
\usepackage{tabularx}
\usepackage{multicol}
\usepackage[subrefformat=parens]{subcaption}
\usepackage{hyperref}

\usepackage{bbm}
\usepackage{bm}
\newcommand{\zissu}{\mathbb{R}} % real number
\newcommand{\sizen}{\mathbb{N}} % natural number
\newcommand{\seisu}{\mathbb{Z}} % integer
\newcommand{\argmax}{\mathop{\mathrm{argmax}}}
\newcommand{\argmin}{\mathop{\mathrm{argmin}}}

\renewcommand{\labelenumi}{(\arabic{enumi})}

\usepackage{url}
\usepackage{breakurl}
\usepackage{comment}
\usepackage{setspace}
\usepackage{stmaryrd}

\numberwithin{equation}{section}

\begin{document}

\title{Sparse solutions of the kernel herding algorithm\\ by improved gradient approximation}

\author{\name Kazuma Tsuji \email kazumatsuji.research@gmail.com \\
       \addr Mitsubishi UFG Bank, Ltd., \\ 2-7-1 Marunouchi, Chiyoda-ku, Tokyo 100-8388, Japan
       \AND
       \name Ken'ichiro Tanaka \email kenichiro@mist.i.u-tokyo.ac.jp \\
       \addr Graduate School of Information Science and Technology\\ 
	The University of Tokyo \\
	PRESTO \\
	Japan Science and Technological Agency (JST) \\ 
	Tokyo, Japan
}

\maketitle

\begin{abstract}%   <- trailing '%' for backward compatibility 
The kernel herding algorithm is used to construct quadrature rules in a reproducing kernel Hilbert space (RKHS). While the computational efficiency of the algorithm and stability of the output quadrature formulas are advantages of this method, the convergence speed of the integration error for a given number of nodes is slow compared to that of other quadrature methods. In this paper, we propose a  modified kernel herding algorithm whose framework was introduced in a previous study and aim to obtain sparser solutions while preserving the advantages of standard kernel herding. In the proposed algorithm, the negative gradient is approximated by several vertex directions, and the current solution is updated by moving in the approximate descent direction in each iteration. We show that the convergence speed of the integration error is directly determined by the cosine of the angle between the negative gradient and approximate gradient. Based on this, we propose new gradient approximation algorithms and analyze them theoretically, including through convergence analysis. In  numerical experiments, we confirm the effectiveness of the proposed algorithms in terms of sparsity of nodes and computational efficiency. Moreover, we provide a new theoretical analysis of the kernel quadrature rules with \emph{fully-corrective weights}, which realizes faster convergence speeds than those of previous studies. 
\end{abstract}


\section{Introduction}
\subsection{Kernel quadrature}
Numerical integration of multivariate functions is indispensable in various fields, including statistics, economics, and physics. In statistical machine learning, numerical integration of multivariate functions is a commonly used tool. Specifically, in Bayesian inference, numerical integration is required in several situations, such as the marginalization of distributions and computation of expectations. 

Kernel quadrature is a type of numerical integration method. It is a quadrature rule for functions in a reproducing kernel Hilbert space\ (RKHS), represented as follows: 

\[ \int_{\Omega} f(x) \mu(\mathrm{d}x) \approx \sum_{i=1}^n \omega_i f(x_i ) \quad \left( \Omega \subset \zissu^d, f \in \mathcal{H}_K , \{ x_i \}_{i=1}^n \subset \Omega, \{ \omega_i \}_{i=1}^n \subset \zissu \right), \]
where $K$ is a positive definite kernel, $\mathcal{H}_K$ is an RKHS, and $\mu$ is a  probability measure. 
 One of the advantages of this method is the flexibility in the choice of an RKHS. By fixing an appropriate RKHS for functions we want to integrate, we can construct an effective integration formula for them. Moreover, it provides an analytically computable way to  optimally construct methods over large families of functions. In addition to the quadrature concept, kernel quadrature is also closely related to other research fields. The Bayesian quadrature \citep{diaconis1988bayesian, o1991bayes} is one such example. It lies in the context of probabilistic numerics \citep{larkin1972gaussian} and is closely related to uncertainty quantification. In the context of Bayesian quadrature, convergence analysis has been studied in \cite{briol2019probabilistic, kanagawa2020convergence}.  Sampling from probability distributions is another related topic (e.g., \cite{10.5555/3023549.3023562, pmlr-v38-lacoste-julien15, briol2017sampling, NEURIPS2019_7012ef03, teymur2020optimal}).  It is known that constructing a quadrature formula that minimizes the integration error is equivalent to minimizing maximum mean discrepancy (MMD) \citep{gretton2012kernel}, which is the distance between probability distributions. Therefore, the research on kernel quadrature provides insights into sampling methods that approximate probability distributions well.

Various methods are used to construct kernel quadrature rules, such as the sequential Bayesian quadrature algorithm \citep{10.5555/3020652.3020694} and orthogonal matching pursuit algorithm \citep{oettershagen2017construction}. In addition, kernel interpolation methods, such as the P-greedy algorithm \citep{de2005near}, can be employed because kernel quadrature rules can be derived by integrating kernel interpolation functions with respect to a measure.  Several methods, including the aforementioned algorithms, compute the optimized weights for a fixed set of nodes in each iteration. Although fast convergence can be expected for the optimized weights, the computation of weights is expensive because the linear equations with the coefficient matrices $(K(x_i, x_j) )_{i , j }$ must be solved in each iteration.  
In addition to these methods, there have been many recent studies on constructing kernel quadrature rules\ \citep{NEURIPS2019_7012ef03, teymur2020optimal, pronzato2021performance, hayakawa2021positively}.

\subsection{Contributions of this study}
In this study, we focus on the kernel herding method \citep{welling2009herding, 10.5555/3023549.3023562, 10.5555/3042573.3042747}, which can be considered as an infinite-dimensional Frank-Wolfe method \citep{fw56}, a continuous optimization method. This method constructs a stable numerical integration formula and the construction procedure is computationally tractable since linear equations do not need to be solved.  However, the convergence speed for the worst-case integration error is low. Exponential convergence has been confirmed in a finite-dimensional RKHS when the step size $\alpha_i$ is determined by line search \citep{beck2004conditional}; however, theoretically, only  $O(\frac{1}{\sqrt{t}})$ convergence has been guaranteed in an infinite-dimensional RKHS \citep{dunn1980convergence}, where $t$ is the number of nodes of the quadrature formula. 

%%%%
This study aims to improve the kernel herding algorithm to derive quadrature rules with \emph{sparse} solutions. Sparsity of nodes is one of the most important factors to assess methods of constructing quadrature rules. To achieve our goal, we approach this problem from the perspective of continuous optimization. We note few studies have approached the improvement of kernel herding in terms of continuous optimization.

\begin{description}
  \item[Section $3$]\ \\
To obtain effective kernel quadrature rules with sparse nodes, we improve the vanilla kernel herding method based on the idea of \cite{pmlr-v119-combettes20a}.
We propose two improved versions of the kernel herding algorithm. The fundamental concept common to both methods involves approximating the negative gradient of the objective function, $F(\nu)= \frac{1}{2} \| \mu_K - \nu \|_K ^2$, by several vertex directions.  We theoretically demonstrate that the convergence speed of the worst-case integration error is directly influenced by $\cos \theta_i$, where $\theta_i$ is the angle between the negative gradient of $F(\nu)$ and $i$-th approximate descent direction. This indicates that the approximation of the negative gradients is significant for constructing a good numerical integration formula. The difference between the two proposed algorithms lies in their approximation methods. 
The first algorithm approximates the negative gradients by positive matching pursuit \citep{locatello2017greedy}. This is a similar procedure to that of \cite{pmlr-v119-combettes20a}. We guarantee the convergence of positive matching pursuit and consider the behavior of $\cos \theta_i$ when $i$ increases. The second algorithm is an improved version of the first algorithm that directly maximizes $\cos \theta_i$, which is a more straightforward approach than the first. We ensure the validity of the maximization method theoretically and provide a convergence guarantee with $O(1/k)$ speed (where $k$ is the number of iterations). In addition, we confirmed through numerical experiments that the proposed algorithms improve the convergence speed of the kernel herding algorithm with respect to the number of nodes and computation time. 

In addition, we propose fully-corrective variants of the proposed methods, where ``fully-corrective'' refers to optimization of the coefficients of the linear combinations of vertex directions. Fully-corrective methods can be realized by solving quadratic optimization. The fully-corrective variants are computationally effective compared to the ordinary fully-corrective kernel herding algorithm.  In the numerical experiments, these algorithms achieved convergence speeds competitive with those of the optimal rates and we confirmed their computational efficiency. The details can be referred to Remark~\ref{remark_difference}.

\item[Section $4$]\ \\
In Section $4$, we focus attention on the great improvement of sparsity by the fully-corrective method from the experiments in Section $3$ and analyze theoretical aspects of fully-corrective kernel quadrature rules. We give a new convergence analysis of the fully-corrective kernel quadrature rules and show the relationship to kernel interpolation; in detail, a fully-corrective kernel quadrature rule with nodes $X= \{x_1, \ldots, x_n  \} \subset \Omega$ achieves a convergence rate at least as fast as the square root of the convergence speed of kernel interpolation with interpolation nodes $X$. While the analysis in previous research can only derive $O(\frac{1}{\sqrt{t}})$ convergence, this analysis can achieve faster convergence rates than those in many cases.
\end{description}

The contributions of this work can be summarized as follows: 
\begin{itemize}
\item We consider improved kernel herding algorithms whose fundamental idea is approximating the negative gradient by several vertex directions. We demonstrate that the convergence speed of the worst-case integration error is directly influenced by $\cos \theta_i$. We propose two algorithms for the approximation of negative gradients. 
In particular, the second novel algorithm directly maximizes $\cos \theta_i$. We provide a theoretical analysis of both algorithms, including convergence analysis.  
\item Through numerical experiments, we confirm that the convergence speed of the proposed algorithms is higher than that of ordinary kernel herding methods with respect to the number of nodes and computation time. Moreover, the fully-corrective variants show significant performance, achieving convergence speeds competitive to optimal rates. 

\item  In Section $4$,  we show the new convergence speed of the kernel quadrature with fully-corrective weights. The convergence speed is beyond the square root rate if the kernel function is sufficiently smooth. Kernel-specific analysis was realized using the theory of kernel interpolation. Although we have not analyzed the algorithm directly, the results give a partial theoretical explanation for the significant performance of the fully-corrective algorithm. 
\end{itemize}

\section{Mathematical background}
\subsection{Problem setting of kernel quadrature}
In this section, we introduce the problem setting of this study. Let $\Omega$ be a subset in $\zissu^d$ and $K : \Omega \times \Omega \to \zissu$ be a positive definite kernel.  We assume the continuity of $K$ and compactness of $\Omega$ in this paper. The function space $\mathcal{H}_{K} (\Omega)$ is the RKHS induced by $K$. For $f_1, f_2 \in \mathcal{H}_K$, we denote the inner product of $f_1$ and $f_2$ by $\left<f_1, f_2\right>_K $ and the norm of $f_1$ by $\| f_1 \|_{K}$. For a Borel probability measure $\mu$ defined on $\Omega$ and $f\in \mathcal{H}_{K} (\Omega)$, we aim to approximate $\int_{\Omega}  f(x)\mu(\mathrm{d}x) $ by a  numerical integration formula $Q_{n}(f) \coloneqq \sum_{i=1}^n \omega_i f(x_i)$, where $X_n=\{ x_i \}_{i=1}^n \subset \Omega$ is a set of nodes and $\{ \omega_i \}_{i=1}^n \subset \zissu$ is a set of weights. We denote the embeddings of the probability measures $\mu$ and discrete measure $\sum_{i=1}^n \omega_i \delta_{x_i}$ to $ \mathcal{H}_{K} $ by
\begin{equation}\label{eq-embed}
\mu_K \coloneqq \int_{\Omega} K(x, \cdot) \mu(\mathrm{d}x)
\end{equation}
 and  $ \nu_K \coloneqq \sum_{i=1}^n \omega_i K(x_i, \cdot)$, respectively.
We evaluate the numerical integration formula, $Q_n$, using the following worst-case error: 
\[ \mbox{(the worst-case error)}\coloneqq
 \sup_{\| f \|_{K} \leq 1 } \left| \int_{\Omega}  f(x)  \mu(\mathrm{d}x)  -  \sum_{i=1}^n \omega_i f(x_i) \right| 
 =\left\| \mu_K - \sum_{i=1}^n \omega_i K(x_i, \cdot) \right\|_K . \]
Note that the second equality can be derived from the Cauchy-Schwarz  inequality. It can be observed that the worst-case error is the distance between $\mu$ and $\sum_{i=1}^n \omega_i \delta_{x_i}$  measured in $\mathcal{H}_{K} $. This distance is called the MMD \citep{gretton2012kernel}. We call the described procedure kernel quadrature.
We note that for fixed nodes $X_n$, it is straightforward to compute the optimized weights $\omega_1, \ldots, \omega_n$ because the squared MMD $\| \mu_K - \sum_{i=1}^n \omega_i K(x_i, \cdot) \|_K ^2$ is a quadratic function of $\omega_1, \ldots, \omega_n$. The optimized weights  can be written as $(\omega_1, \ldots, \omega_n)^\top={ K_{X_n}} ^{-1} z_{X_n}$, and the quadrature rule for $f \in  \mathcal{H}_K$ is
\begin{equation}\notag
f_{X_n} ^\top K_{X_n} ^{-1} z_{X_n}, 
\end{equation}
where  $K_{X_n} = (K(x_i, x_j))_{1\leq i, j \leq n}, z_{X_n} = (\mu_K (x_1), \ldots, \mu_K(x_n))^\top$, and $f_{X_n}= (f(x_1), \ldots, f(x_n) )^\top$. In addition, we can calculate the worst-case error as follows:
 \begin{equation}\notag
 \sup_{\| f \|_{K} \leq 1 } \left| \int_{\Omega}  f(x)  \mu(\mathrm{d}x)  -  \sum_{i=1}^n \omega_i f(x_i) \right| = \int \int K(x, y) \mu(\mathrm{d}x) \mu(\mathrm{d}y) -  z_{X_n} ^{\top} K_{X_n}^{-1} z_{X_n}. 
 \end{equation}
 
  These optimized weights $\{\omega_i\}_{i=1}^n$ are often used and important when considering the relationships to other research fields that we mention later. 

 \subsection{Kernel herding}
Kernel herding \citep{welling2009herding, 10.5555/3023549.3023562, 10.5555/3042573.3042747} is a commonly used method for kernel quadrature. This method constructs a quadrature rule by solving the optimization problems in an RKHS. In the following, we introduce kernel herding and describe its algorithm in detail. 
 
First, we introduce some notations. For a set $S$, we denote the convex hull and conical hull of $S$ by $\mathrm{conv}(S)$ and $\mathrm{cone}(S)$, respectively. That is, 
\[ \mathrm{conv}(S) = \left\{ \sum_{i=1}^k  c_i  s_i \middle|\  c_1, \ldots c_k \geq 0, \sum_{i=1}^k c_i =1, s_1, \ldots, s_k \in S, k \geq1 \right\} \] and 
\[\mathrm{cone}(S) = \left\{ \sum_{i=1}^k  c_i  s_i \middle|\  c_1, \ldots c_k \geq 0, s_1, \ldots, s_k \in S, k \geq1 \right\}.\]
In addition, for an element $a$ and set $S$, we define $S - a$ as
\[ S - a \coloneqq \{ s - a \mid s\in S  \}.  \]

Let $V \coloneqq \{ K(x, \cdot) \mid x \in \Omega \}$ and $M \coloneqq \overline{\mathrm{conv}(V)}$, where the closure is taken with respect to $\|\cdot \|_K$.

Here, we assume that $\mu_K$ defined in (\ref{eq-embed}), which is the embedding of a Borel probability measure $\mu$, belongs to $M$. This assumption is very weak. Lemma~\ref{proof of lemma} in \autoref{Proofs} shows that $\mu_{K} \in M$ under the assumption of the continuity of $K$ and compactness of $\Omega$.

Let $F(\nu) \coloneqq  \frac{1}{2}  \| \mu_K - \nu  \|_K ^2$. The problem setting  of kernel quadrature is rewritten as the minimization problem of $F$ for the embeddings of discrete measures to $\mathcal{H}_K$. By simple calculation, we can derive the Fr\'{e}chet derivative $\nabla_{\nu} F (\nu)  = \nu - \mu_K $, which is computed with respect to the metric of $\mathcal{H}_K$.
To minimize $F(\nu)$, the following gradient descent method can be considered:
$\nu_{t+1} \coloneqq \nu_t - \gamma \nabla_{\nu} F(\nu_t) = \nu_t + \gamma (\mu_K - \nu_t)\quad(\gamma > 0).$
However, because the descent direction $\mu_K - \nu$ is not necessarily of the form $\sum_{i=1}^n c_i K(x_i, \cdot)$, gradient descent is not suitable for deriving a numerical integration formula. It is appropriate if we can select a descent direction of the form 
\begin{equation}\label{eq2-1}
 K(z_i, \cdot) - \nu_t \ (z_i \in \Omega)
 \end{equation}
in some manner. Kernel herding is a typical method based on this principle.
In the kernel herding algorithm, to search a descent direction of the form of (\ref{eq2-1}), we consider the maximization of
 \begin{equation}\label{eq2-2}
 \left< \mu_K - \nu_t , v - \nu_t \right>_K 
 \end{equation}
 subject to $v \in V = \{ K(z, \cdot) \mid z \in \Omega \}$.
 
 The algorithm is given in \autoref{algo1}. In each iteration, a point $v_{t+1} \in V$ is selected to maximize the inner product (\ref{eq2-2}); it moves from $\nu_t$ in the direction of $v_{t+1}$ with  step size $\alpha_t$. The step size $\alpha_t$ is usually determined by $\alpha_t = \frac{1}{t+1}$ or line search: $\alpha_t = \argmin_{0 \leq \alpha \leq1} \| (1-\alpha)\nu_t + \alpha  v_{t+1}\|_K $. Eventually, the output $\nu_n$ has the form $\sum_{i=1}^n \omega_i K(x_i, \cdot)$, which corresponds to the numerical integration formula $\sum_{i=1}^n \omega_i f(x_i)$.
 
    \begin{algorithm}[H] 
 \caption{Kernel herding}
 \label{algo1}
 \begin{algorithmic}[1]
 \renewcommand{\algorithmicrequire}{\textbf{Input:}}
 \renewcommand{\algorithmicensure}{\textbf{Output:}}
\STATE select an initial point $\nu_1  \in V$
  \FOR {$t = 1$ to $n-1$}
  \STATE $v_{t+1}= \argmax_{\nu\in V} \left<\mu_K - \nu_t, v - \nu_t \right>_K$
    \STATE determine the step size $0 < \alpha_t \leq 1$
  \STATE $\nu_{t+1} = (1 - \alpha_t)\nu_t + \alpha_t v_{t+1}$
  \ENDFOR
 \RETURN $\nu_n$
 \end{algorithmic} 
 \end{algorithm}
  
 \begin{remark}
 When $\nu_t (x)  = \sum_{j=1}^n \omega_i K(x, x_j)$ and $v(x)= K(x, z_i)$, the inner product in (\ref{eq2-2}) is given in the following form:
 \begin{align*}
  \left< \mu_K - \nu_t , v - \nu_t \right>_K &=  \int_{\Omega} K(z_i , y) \mu(\mathrm{d}y) - \sum_{j=1}^n \omega_j K(z_i, x_j)  \\
  &\ -\sum_{j=1}^n \omega_j  \int_{\Omega} K(x_j, y) \mu(\mathrm{d}y) + \sum_{i=1}^n \sum_{j=1}^n \omega_i \omega_j K(x_i, x_j).
  \end{align*}
 Therefore, this value must be maximized with respect to $z_i \in \Omega$. Because this maximization problem is not convex, we prepare several candidate points and select the maximum point from them. 
 \end{remark}

Kernel herding outputs a stable quadrature rule with $\sum_{i=1}^n \omega_i =1$ and $\omega_i \geq 0\ (i=1, \ldots, n)$. Stability of the quadrature rule corresponds to boundedness of the $\ell_1$ norm of $(\omega_1, \ldots, \omega_n)^\top$ 
\[ \sum_{i=1}^n |  \omega_i | < C , \] 
where $C$ is a positive constant. This is an advantage because it is known that some construction methods of kernel quadrature result in an unstable quadrature formula \citep{oettershagen2017construction}; moreover, there are not many methods of kernel quadrature whose numerical stability is guaranteed theoretically. In addition, we do not need to solve the linear equation with coefficient matrix $(K(x_i, x_j))_{1\leq i, j\leq t}$ in each iteration.

% Acknowledgements should go at the end, before appendices and references
\subsection{Frank-Wolfe algorithm and its variants}
The kernel herding algorithm can be considered as an infinite-dimensional Frank-Wolfe algorithm \citep{fw56}, which is a convex optimization method in Euclidean space. In this section, we describe the Frank-Wolfe algorithm and its variants. 
\subsubsection{Frank-Wolfe algorithm}
The Frank-Wolfe algorithm \citep{fw56} (a.k.a. Conditional Gradients \citep{polyak66cg} ) is an important class of first-order methods for constrained convex minimization, i.e., solving 
$$\min_{x \in C} f(x),$$ 
where $C \subset \zissu^d$ is a compact convex feasible region. We denote the vertices of $C$ by $V_C$, which means $C = \mathrm{conv}(V_C)$. The algorithm is shown in \autoref{Frank-Wolfe algorithm}. These methods usually form their iterates as convex combinations of feasible points, and as such, they do not require (potentially expensive) projections onto the feasible region $C$. Moreover, the access to the feasible region is solely realized by means of a so-called \emph{linear minimization oracle (LMO)}, which  returns $\arg\min_{x \in P} c^\intercal x$ upon presentation with a linear function $c$. Another significant advantage is that the iterates are typically formed as \emph{sparse} convex combinations of extremal points of the feasible region (sometimes also called \emph{atoms}), which makes this class of optimization algorithms particularly appealing for problems such as sparse signal recovery, structured regression, SVM training, and kernel herding.

\begin{algorithm}[H]
 \caption{Frank-Wolfe}
 \label{Frank-Wolfe algorithm}
 \begin{algorithmic}[1]
 \REQUIRE initial vertex $\xi_1 \in  V_C$,  set of vertices $S_1 = \{ \xi_1 \}$
   \FOR {$t = 1$ to $n-1$}
  \STATE $v_{t+1} = \argmax_{v\in  V_C} \left< -\nabla f(\xi_t), v - \xi_t \right>$   
    \STATE determine the step size $0 < \alpha_t \leq 1$ 
  \STATE $\xi_{t+1} = (1 - \alpha_t)\xi_t + \alpha_t  v_{t+1}$
  \STATE $S_{t+1}= S_t \cup \{ v_{t+1} \}$
  \ENDFOR
 \RETURN $\xi_n$ 
 \end{algorithmic} 
 \end{algorithm}

\subsubsection{Variants of the Frank-Wolfe algorithm}
There have been various studies on the modification of the vanilla Frank-Wolfe algorithm that aim to achieve faster convergence speed or sparser solutions. One of the famous variants is the away-step Frank-Wolfe method \citep{wolfe70}. In this algorithm, we choose the direction from the ordinary direction $v_{t+1} - \xi_{t}$ and   away-direction $\xi_t - a_{t+1}$, where $a_{t+1} = \argmin_{v\in S_t} \langle - \nabla f(\xi_t), v - \xi_t \rangle$, and we move in the direction in each step. It has been shown that this modification improves the convergence speed and sparsity of solutions. The pairwise Frank-Wolfe algorithm \citep{lacoste15} is another famous example. In this algorithm, the direction $v_{t+1} - a_{t+1}$ is used instead of $v_{t+1} - \xi_{t+1}$.

The fully-corrective Frank-Wolfe algorithm (e.g., \cite{holloway1974extension, jaggi13fw}) is another modified Frank-Wolfe algorithm. It executes the optimization of the point $\xi_t$ over the convex hull  of the active atoms in each iteration. The algorithm is \autoref{fc-cg}.  We note that  the optimization problem on line \ref{opt-phase} of \autoref{fc-cg}  is a convex optimization problem because $f$ is a convex function. Although it is somewhat computationally expensive when the size of $S_t$ increases, we can expect sparse solutions. We note that the nice sparsity of the solutions of the fully-corrective kernel herding algorithm was confirmed experimentally in previous studies \citep{10.5555/3042573.3042747,pmlr-v38-lacoste-julien15}.

\begin{algorithm}[H]
 \caption{fully-corrective Frank-Wolfe method}
 \label{fc-cg}
 \begin{algorithmic}[1]
 \REQUIRE initial vertex $v_1=\xi_1 \in  V$,  set of vertices $S_1 = \{ \xi_1 \}$
   \FOR {$t = 1$ to $n-1$}
  \STATE $v_{t+1} = \argmax_{v\in  V_C} \left< -\nabla f(\xi_t), v - \xi_t \right>$   
    \STATE determine the step size $0 < \alpha_t \leq 1$ 
  \STATE $\xi_{t+1} = (1 - \alpha_t)\xi_t + \alpha_t  v_{t+1}$
  \STATE $S_{t+1}= S_t \cup \{ v_{t+1} \}$
  \STATE $ w(v_{1}), \ldots, w(v_{t+1}) \leftarrow \argmin_{w_{v_1}, \ldots, w_{v_{t+1} \in \zissu}} f(\sum_{i=1}^{t+1} w_{v_i} v_i)$, s.t. $\sum_{i=1}^{t+1} w_{v_i} =1, w_{v_i} \geq 0\ (i=1, \ldots t+1)$  \label{opt-phase}
  \STATE $\xi_{t+1}=\sum_{i=1}^{t+1} w(v_i) v_{i}$
  \ENDFOR 
 \RETURN $\xi_n$ 
 \end{algorithmic} 
 \end{algorithm}

\section{Acceleration of kernel herding by approximating negative gradients}\label{chapter3}

 Although the kernel herding algorithm  (\autoref{algo1}) outputs stable quadrature rules and is computationally tractable, its convergence speed of integration errors is slower than that of other effective methods. This has been confirmed experimentally, for example, in \cite{10.5555/3042573.3042747, 10.5555/3020652.3020694}. The cause of this problem is considered to be the zig-zagging trajectory (\autoref{zigzag}), which is often featured as the cause of slow convergence of the vanilla Frank-Wolfe algorithm.  This zig-zagging trajectory arises in the Frank-Wolfe method when a point $\xi_t$ moves towards the vertex in each iteration and the direction does not approximate the negative gradient $-\nabla f(\xi_t)$ efficiently in some cases. \autoref{decrease_cos} suggests that a similar phenomenon occurs in the kernel herding case.  It becomes difficult to approximate the negative gradient $\mu_K - \nu_t$ using only one direction $v_{t+1} - \nu_t$ as $t$ increases. \autoref{decrease_cos}  shows the decrease of $\cos \theta_t$ as $t$ increases, where $\theta_t$ is the angle between $\mu_K - \nu_t$ and $v_{t+1} - \nu_t$. It can be observed from \autoref{decrease_cos} that the approximation of the negative gradient $\mu_K - \nu_t$ by $v_{t+1} - \nu_t$ is insufficient.

To overcome this, several methods have been examined, such as the away-step method and pairwise method, as mentioned above. In this study, we focused on \cite{pmlr-v119-combettes20a}. The main idea of \cite{pmlr-v119-combettes20a} is  approximating the negative gradient $- \nabla f(\xi_t)$, by a linear combination of vertex directions with positive coefficients, that is,  solving  $\min_{d\in \mathrm{cone}( V - \xi_t ) }\| -\nabla f(\xi_t) -d \|^2$ in each iteration. \autoref{combet} is an example; $-\nabla f(\xi_t)$ is approximated by $d_2$, which is the linear combination of the two vertex directions. This helps avoid zig-zagging trajectories. For the constrained convex optimization problems in Euclidean space, the proposed algorithms improve convergence speeds with respect to the number of iterations and computation time.

  \begin{figure}[h]
    \begin{minipage}[t]{0.45\hsize}
        \center
        \captionsetup{width=.95\linewidth}
        \includegraphics[width=\textwidth]{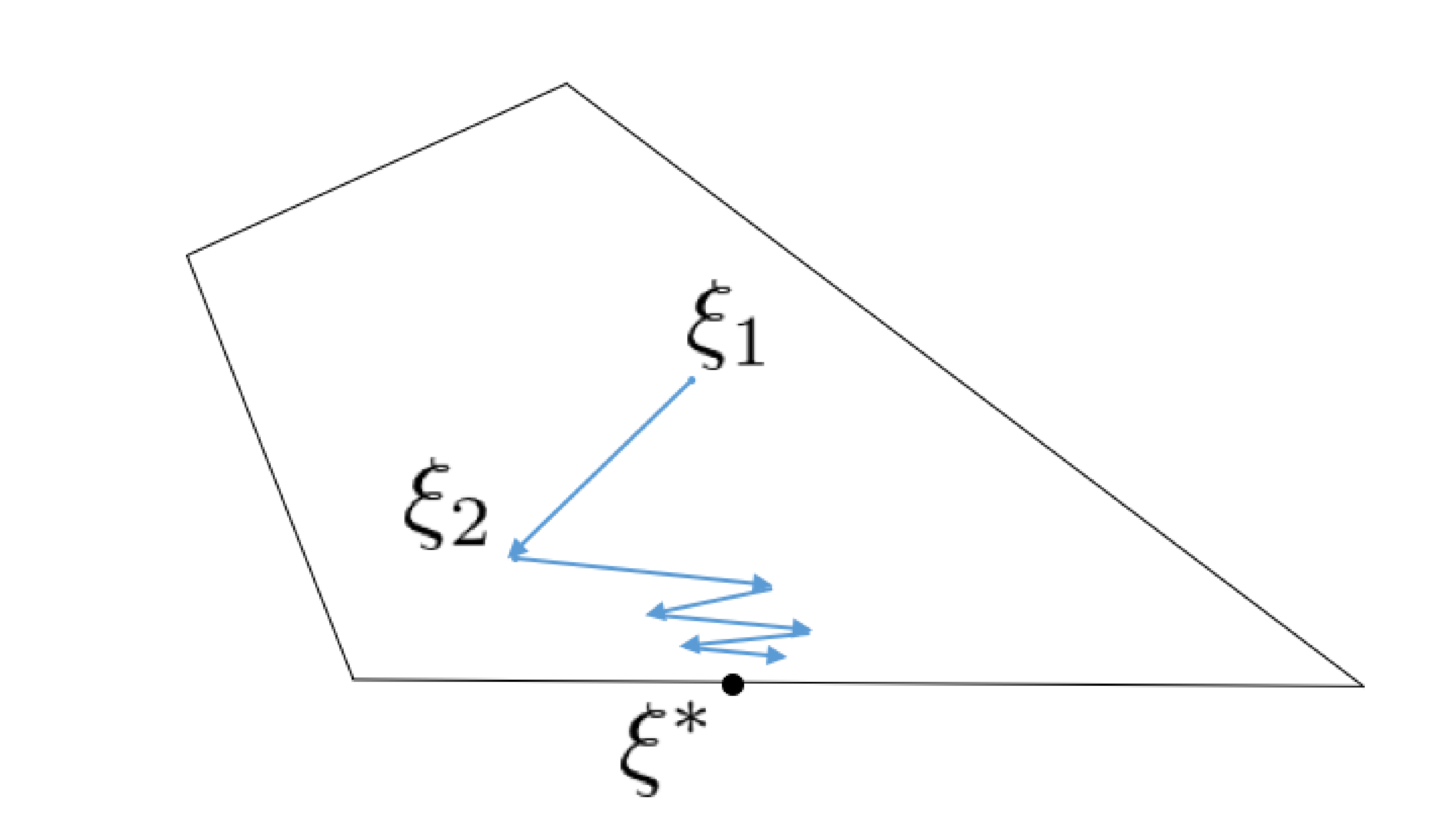}
        \caption{Zig-zagging trajectory in the Frank-Wolfe method}
        \label{zigzag}
    \end{minipage}
    \begin{minipage}[t]{0.45\hsize}
        \center
        \captionsetup{width=.95\linewidth}
        \includegraphics[width=\textwidth]{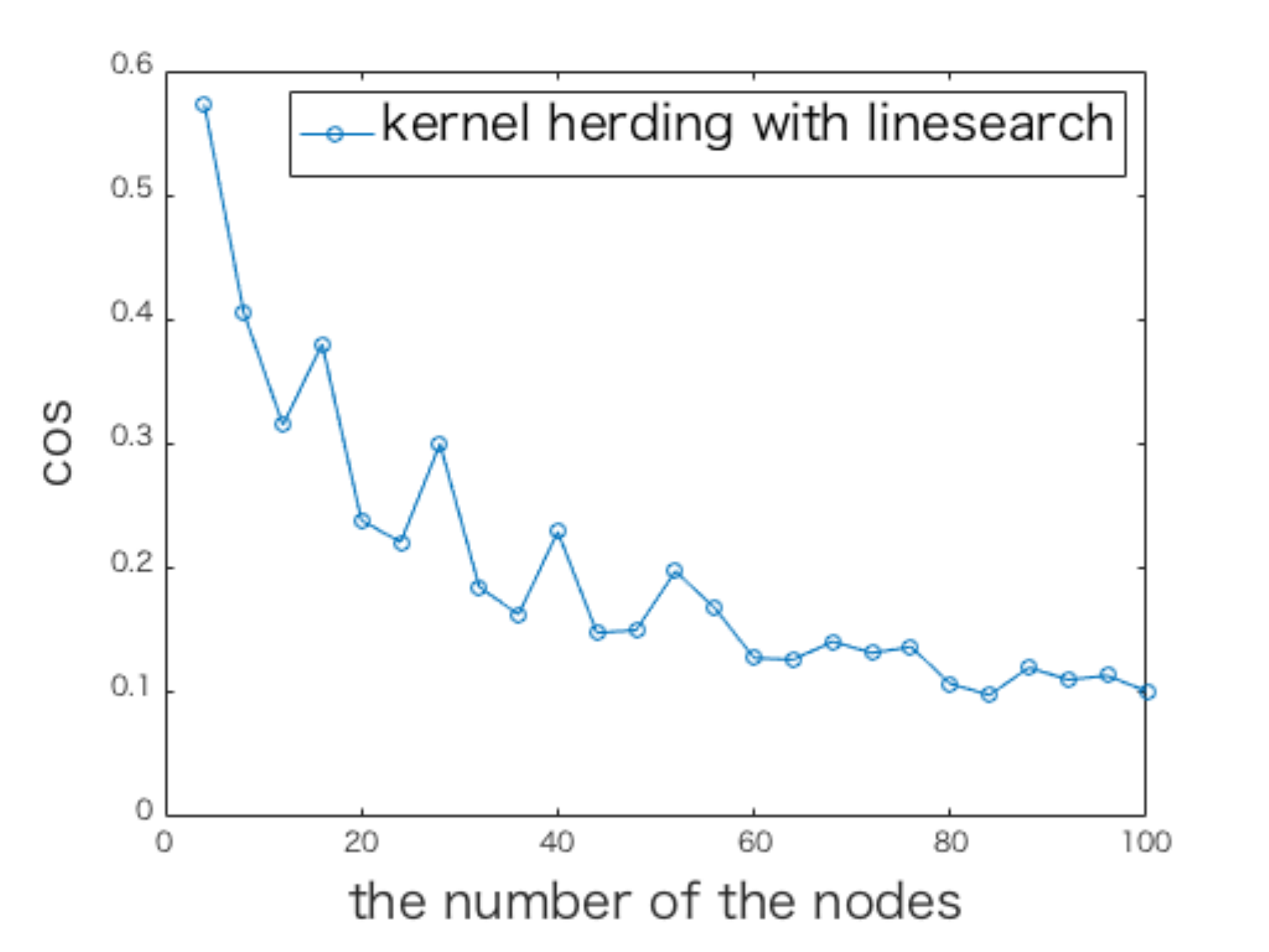}
        \caption{Decrease in $\cos$ between $\nu_t - \mu_K$ and $v_{t+1} - \nu_t$}
        \label{decrease_cos}
         \end{minipage}  
      \begin{minipage}[t]{0.45\hsize}
  \begin{center} %センタリングする
    \captionsetup{width=.95\linewidth}
        \includegraphics[width=\textwidth]{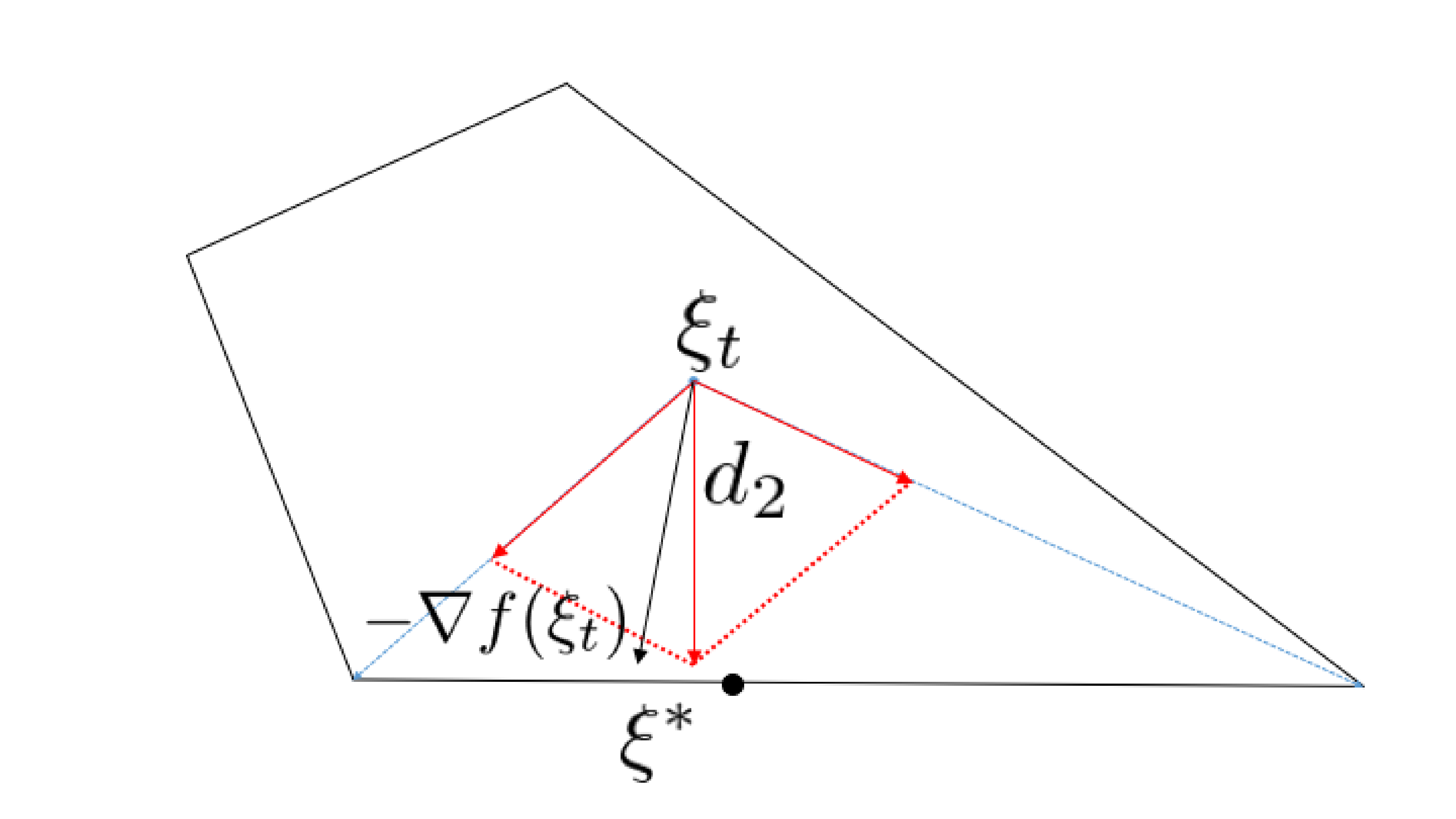}
    \caption{Gradient approximation idea proposed in \cite{pmlr-v119-combettes20a}} %タイトルをつける
    \label{combet} %ラベルをつけ図の参照を可能にする
  \end{center}
   \end{minipage}
\end{figure}

\subsection{New kernel herding algorithms}\label{New algorithms of kernel herding}

In this section, we present the improved version of the kernel herding algorithm, given in \autoref{boost herding}. The algorithm approximates the negative gradient $-\nabla_\nu F(\nu_t) =  \mu_K - \nu_t$ by $g_t$, which is a component of $\mathrm{cone}(V - \nu_t)$, and updates $\nu_t$ with $\nu_{t+1} = \nu_t + \gamma_t g_t$. This concept has already been used to improve the Frank-Wolfe method in \cite{pmlr-v119-combettes20a}. The aim of this study was to derive quadrature rules with \emph{sparse} nodes by \autoref{boost herding}.
This is the main difference from \cite{pmlr-v119-combettes20a}, which did not consider the sparsity of solutions. 
To derive sparse solutions, the algorithm for the approximation of negative gradients (line  \autoref{approximation_of_ng} of \autoref{boost herding}) is important. In this study, we use the approximation algorithms  \autoref{positive matching pursuit} or \autoref{greedy cos}. In \autoref{subs3-1}, we discuss \autoref{boost herding}, which uses \autoref{positive matching pursuit}. The algorithm is a natural extension of the improved Frank-Wolfe method proposed in \cite{pmlr-v119-combettes20a}.  In \autoref{subs3-2}, we improve the algorithm of \autoref{subs3-1}. This algorithm mainly involves the maximization of $\cos \theta_t$, where $\theta_t$ is the angle between the negative gradient $-\nabla_\nu F(\nu_t)$ and approximate gradient.

\begin{algorithm}[H]
 \caption{Accelerated kernel herding algorithm}
 \label{boost herding}
 \begin{algorithmic}[1]
 \renewcommand{\algorithmicrequire}{\textbf{Input:} }
 \renewcommand{\algorithmicensure}{\textbf{Output:}} 
\REQUIRE   Input point $\nu_1 \in V$
\ENSURE   $\nu_{T} \in M$

  \FOR {$t = 1$ to $T-1$}
   \STATE  derive the approximate gradient $g_t \in \mathrm{cone}(V - \nu_t)$ using \autoref{positive matching pursuit} or \autoref{greedy cos} \label{approximation_of_ng}
   \STATE determine the step size $\gamma_t \in (0, 1)$ by line search
   \STATE $\nu_{t+1} \leftarrow \nu_t + \gamma_t g_t$
   \ENDFOR
 \RETURN $\nu_T$ 
 \end{algorithmic} 
 \end{algorithm}
 
 \begin{algorithm}[H]
 \caption{Positive matching pursuit}
 \label{positive matching pursuit}
 \begin{algorithmic}[1]
 \renewcommand{\algorithmicrequire}{\textbf{Input:}  }
 \renewcommand{\algorithmicensure}{\textbf{Output:}} 
\REQUIRE   Input $\nu_t \in M$, $d_0 =0$, maximum number of rounds $K_{\max} \in \sizen$, truncation parameter $\delta < 1$, and $\Lambda_t=0$.

\ENSURE    approximate direction $g_t$

    \FOR {$k=0 $ to $K_{\max}-1$}
  \STATE $r_k \leftarrow  -(\nu_t - \mu_K)- d_k$
  \STATE $v_k \leftarrow \argmax_{v\in V} \left< r_k, v \right>_K$
  \STATE $u_k \leftarrow \argmax_{u\in \{ v_k - \nu_t, -d_k / \| d_k \|_K \}  } \left< r_k, u \right>_K$\ (if $k=0$, $\left< r_k, -d_k / \| d_k \|_K\right>_K \coloneqq -\infty$) \label{argmax_u}
  \STATE $ d_k ^{'} = d_k + \frac{\left< r_k, u_k \right>_K}{\| u_k \|_K ^2} u_k $\ ($\lambda_k \coloneqq \frac{\left< r_k, u_k \right>_K}{\| u_k \|_K ^2} $)
  \IF{$\mathrm{align}(-(\nu_t -\mu_K), d_k ^{'}) - \mathrm{align}(-(\nu_t -\mu_K), d_k ) > \delta$ } \label{compare align}
 \STATE $d_{k+1} \leftarrow d_k ^{'}$
 \STATE $\Lambda_t \leftarrow $\\
 $
             \begin{cases} 
             \Lambda_t +\lambda_k  \quad(  \mbox{if}\ u_k = v_k - \nu_t) \\
             \Lambda_t ( 1  -\lambda_k/\| d_k \| ) \quad( \mbox{if}\ u_k = \frac{-d_k} { \| d_k \|_K})
             \end{cases}$
\ELSE 
\STATE  \textbf{break}
  \ENDIF \label{line11}
  \ENDFOR
  \STATE $K_t \leftarrow k$
    \RETURN $g_t \coloneqq d_k / \Lambda_t$ 
 \end{algorithmic} 
 \end{algorithm}

\begin{algorithm}[H]
 \caption{Greedy maximization of $\cos \theta_t$}
 \label{greedy cos}
 \begin{algorithmic}[1]
 \renewcommand{\algorithmicrequire}{\textbf{Input:}  }
 \renewcommand{\algorithmicensure}{\textbf{Output:}} 
\REQUIRE   Input $\nu_t \in M$, maximum number of rounds $K_{\max} \in \sizen$, and truncation parameter $\delta \geq 0$. 

\ENSURE    approximate direction $g_t$
\STATE $d_1 \leftarrow \argmax_{v\in V-\nu_t} \frac{ \left< -(\nu_t - \mu_K),   v \right>_K}{\| v \|_K \| \nu_t - \mu_K \|_K}$
\STATE $\Lambda_t \leftarrow 1$
    \FOR {$k=1 $ to $K_{\max}-1$}
  \STATE $c_k, v_k \leftarrow  $\\$\argmax_{c \geq 0, v\in V - \nu_t} \frac{\left< -( \nu_t - \mu_K), d_k + c v \right>_K}{\| \mu_K -\nu_t \|_K \| d_k +c v  \|_K}$ \label{line1}
  \STATE $ d_k ^{'} = d_k + c_k v_k$
  \IF{$\mathrm{align}(-(\nu_t -\mu_K), d_k ^{'}) - \mathrm{align}(-(\nu_t -\mu_K), d_k ) > \delta$ }
 \STATE $d_{k+1} \leftarrow d_k ^{'}$
 \STATE $\Lambda_t \leftarrow \Lambda_t + c_k $
\ELSE 
\STATE  \textbf{break}
  \ENDIF
  \ENDFOR
  \STATE $K_t \leftarrow k$
    \RETURN $g_t \coloneqq d_k / \Lambda_t$ 
 \end{algorithmic} 
 \end{algorithm}

\subsubsection{Positive matching pursuit}\label{subs3-1}
 We define $\mathrm{align}(f_1, f_2)$ for $f_1, f_2 \in \mathcal{H}_K (\Omega)$ as follows: for $f_1, f_2 \neq 0$, $\mathrm{align}(f_1, f_2) \coloneqq  \frac{\left< f_1, f_2\right>_K}{\| f_1 \|_K \| f_2 \|_K}$; otherwise, $\mathrm{align}(f_1, f_2)= -1$. 
 
In the following, we discuss \autoref{boost herding}, which uses positive matching pursuit (\autoref{positive matching pursuit}). The positive matching pursuit algorithm was introduced in \cite{locatello2017greedy}. This can be used as  a greedy approximation algorithm for an element in an inner product space. The target is approximated by a positive linear combination of elements in a cone. On line \autoref{argmax_u}, the backward direction $-d_k/\| d_k \|_K$ may sometimes be selected as a modifier of the approximate direction.  In addition, on line \autoref{compare align}, we check whether the updated direction increases $\cos \theta_t$, where $\theta_t$ is the angle between the approximate direction and $-\nabla_{\nu} F(\nu_t)$. If the increase is smaller than $\delta$, the iteration is stopped. In general, $\cos \theta_t$ does not increase monotonically. Therefore, $\delta$ can take negative values.

In \cite{pmlr-v119-combettes20a}, $K_{\max}$ was taken as a significantly large number, and only the convergence speed with respect to $T$ was  considered. However, in this study, the total number of iterations must be managed, including $\sum_{i=1}^T K_i$, which is an upper bound of the total number of nodes generated by kernel herding. This is the main difference between this work and \cite{pmlr-v119-combettes20a}. If $K_{\max}$ is too large, the negative gradient $-\nabla_\nu F(\nu_t)$ is approximated well, but the convergence speed for a given number of nodes may be slow. Conversely, if $K_{\max}$ is too small, $g_t$ cannot approximate $-\nabla_\nu F(\nu_t)$ sufficiently; thus, the convergence speed may not differ from that of ordinary kernel herding. Therefore,  $K_{\max}$ should be selected carefully.

In this subsection, we theoretically analyze \autoref{boost herding} that implements \autoref{positive matching pursuit}. Some results from \cite{pmlr-v119-combettes20a} can be referred to; however, \cite{pmlr-v119-combettes20a} only considers the Frank-Wolfe method in Euclidean space. 

Let us introduce some notations. For $a, b \in \seisu$, we denote  $\{ i \in \seisu \mid  a\leq i \leq b\}$ by $ \llbracket a,  b\rrbracket$. In addition, we define 
 $\epsilon_t \coloneqq  \frac{1}{2} \| \mu_K - \nu_t \|_K ^2$. This is the square of the worst-case error at the $t$-th iteration.

First, we confirm that the algorithm works as required following a similar argument to that of Proposition 3.1 in \cite{pmlr-v119-combettes20a}. 
 \begin{proposition}\label{prop1}
 Let $t \in \llbracket 1,   T-1\rrbracket$ and $\nu_t \in M$. Then, \\
 (i)\  $d_1$  is well-defined and  $K_t \geq 1$, \quad (ii)\ $\lambda_0, \ldots, \lambda_{K_t -1} \geq 0$,\\
(iii)\ for $k\in \llbracket 0, K_t \rrbracket, d_k \in \mathrm{cone}(M - \nu_t )$, \quad (iv)\ $\nu_t +g_t \in M \ \mathrm{ and }\  \nu_{t+1} \in M$,\\
and (v)\ $\mathrm{align} ( -(\nu_t - \mu_K), g_t ) \geq \mathrm{align} ( -(\nu_t - \mu_K), v_0 -  \nu_t )+ (K_t -1)\delta$, where $v_0 \in \argmax_{v\in M} \left< -(\nu_t - \mu_K), v \right>$. In addition, $ \mathrm{align}( -(\nu_t - \mu_K), v_0 - \nu_t ) \geq 0$.
 \end{proposition}
  \begin{proof}
 (i)\ By definition, $\mathrm{align}(r_0, d_0)= -1$. In addition, because $v_0 \in \argmax_{v\in M} \left< r_0, v \right>_K$ and $\nu_t \in M$, the following holds true: 
 \begin{equation}\label{eq1}
  \left< r_0, v_0 - \nu_t  \right>_K = \left< r_0, v_0 \right>_K -   \left< r_0, \nu_t \right>_K \geq 0 .
  \end{equation}
  Therefore, $u_0 = v_0 - \nu_t$ and $\mathrm{align}(r_0, d_0 ^{'}) - \mathrm{align}(r_0, {d_0}) \geq 1 > \delta$ are valid. Thus, $K_t \geq 1$. \\
  
 (ii)\  Let $k \in  \llbracket 0, K_t -1  \rrbracket$. Because $v_k \in \argmax_{v \in M} \left<  r_k, v \right>_K$ and $\nu_t \in M$, 
 \[ \left< r_k, v_k -\nu_t \right>_K = \max_{v \in M} \left< r_k, v \right>_K -  \left< r_k, \nu_t \right>_K \geq 0 .\]
 On line \ref{argmax_u} of \autoref{positive matching pursuit}, 
 \[ \left< r_k,  u_k\right>_K \geq   \left< r_k, v_k - \nu_t \right>_K \geq 0  .\]
 Thus, $\lambda_k \geq 0$. 
\\
(iii) We prove this by induction: \\
(a)\ $k=0$\\
 $d_0 = 0 = \nu_t -\nu_t \in  \mathrm{cone}(M -  \nu_t )$. \\
 (b)\ We assume there exist $k \in \sizen$ and $d_k \in \mathrm{cone}(M - \nu_t )$. \\
 If $u_k = v_k - \nu_t$, because $u_k \in M -\nu_t$ and $\lambda_k \geq 0$, 
 \[ d_{k+1} = d_k + \lambda_k (v_k -\nu_t)  \in \mathrm{cone}(M - \nu_t ).  \]
Consider $u_k = -\frac{d_k}{\| d_k \|_K}$. Because $d_{k+1} = (1 - \frac{\lambda_k}{\| d_k \|_K} )d_k$, we must prove that $ (1 - \frac{\lambda_k}{\| d_k \|_K} ) \geq 0$. It suffices to show that $ (1 - \frac{\lambda_k}{\| d_k \|_K} ) \geq \frac{1}{2}$. By simple calculation, 
 \begin{align}
 1 - \frac{\lambda_k}{\| d_k \|_K} \geq \frac{1}{2} &\Leftrightarrow \frac{1}{2}\geq \frac{\lambda_k}{\| d_k \|_K} = \frac{\left< r_k, -d_k /\| d_k \|_K \right>_K}{\| d_k \|_K} \notag\\
&\Leftrightarrow \frac{\| d_k \|^2}{2} \geq \left< r_k, -d_k \right> .\notag
 \end{align}
Thus, it suffices to show that $ \frac{\| d_k \|^2}{2} \geq \left< r_k, -d_k \right>$. The following inequality is valid for any $k^{'} \in \sizen$: 
 \begin{align*}
 \| r_{k^{'}+1} \|_K ^2
 &= \| -(\nu_t - \mu_K) -d_{k^{'}+1} \|_K ^2  \\
 &=  \| -(\nu_t - \mu_K) - (d_k^{'} + \lambda_{k^{'}} u_{k^{'} }) \|_K ^2 \\
  &= \| r_{k^{'}} - \lambda_{k^{'}} u_{k^{'}} \|_K ^2  \\
 &= \| r_{k^{'}}  \|_K ^2 -2\lambda_{k^{'}} \left< r_{k^{'}}, u_{k^{'}} \right>_K +\lambda_{k^{'}} ^2 \| u_{k^{'}} \|_K ^2 \\
 &= \| r_{k^{'}} \|_K ^2 -\frac{\left< r_{k^{'}}, u_{k^{'}} \right>_K ^2 }{\| u_{k^{'}} \|_K ^2}\\
 &\leq \| r_{k^{'}} \|_K ^2 .
 \end{align*}
Note that we used $\lambda_{k^{'}} = \frac{\left< r_{k^{'}}, u_{k^{'}} \right>}{\| u_{k^{'}} \|^2}$. Therefore, 
 \begin{equation}\label{eq2}
 \| r_k \|_K ^2 \leq \| r_0 \|_K ^2.
 \end{equation}
Because $d_0 = 0$, by (\ref{eq2}), it is evident that
  \begin{align*}
   \| -(\nu_t - \mu_K ) \|_K ^2  &\geq  \| -(\nu_t -\mu_K) - d_k \|_K ^2 = \|  \nu_t - \mu_K \|_K ^2 + 2\left< \nu_t -\mu_K, d_k \right>_K + \| d_k \|_K ^2  \\
    \Leftrightarrow \left< \nu_t -\mu_K, d_k \right>_K &\leq -\frac{\| d_k\|_K ^2 }{2} .
  \end{align*}
Using this, 
  \begin{align}
  \left< r_k , - d_k \right>_K &= \left< -( \nu_t - \mu_K  ) - d_k, - d_k \right>_K \notag\\
  &= \left< \nu_t - \mu_K, d_k  \right>_K + \| d_k \|_K ^2 \notag\\
 & \leq \frac{\| d_k \|_K ^2}{2} \notag .
  \end{align}
 Therefore, $d_{k+1} \in \mathrm{cone}(M - \nu_t )$. 
  \\
  (iv)\ By (iii), $d_{K_t} \in \mathrm{cone}(M - \nu_t )$. In addition, from the algorithm and (ii), it is evident that $g_t \in M - \nu_t$. Therefore, 
    $ \nu_t +g_t \in M$ and 
  \[ \nu_{t+1} = \nu_t + \gamma_t g_t = (1-\gamma_t)\nu_t + \gamma_t (\nu_t + g_t ) \in M . \]  
  (v) By (\ref{eq1}), $\mathrm{align}( -(\nu_t - \mu_K), v_0 -\nu_t ) \geq 0$. In addition, by $g_t = d_{K_t} / \Lambda_t$ and line \ref{compare align} of \autoref{positive matching pursuit}, the following holds true: 
  \begin{align*}
  \mathrm{align}(-(\nu_t -\mu_K), g_t) &=  \mathrm{align}(-(\nu_t -\mu_K), d_{K_t} ) \\
  &\geq  \mathrm{align}(-(\nu_t -\mu_K), d_{K_t-1} ) + \delta \\
    &\geq  \mathrm{align}(-(\nu_t -\mu_K), d_1 ) +(K_t -1) \delta \\
    &=\mathrm{align}(-(\nu_t -\mu_K), v_0 - \nu_t ) +(K_t -1)\delta .
  \end{align*}
   \end{proof}

By Proposition~\ref{prop1}, we can ensure that the output of the algorithm $\nu_t$ can be written as the convex combination of kernel functions $\{ K(x_i, \cdot) \}_{i=1}^n$. Therefore, \autoref{boost herding} outputs a stable quadrature  formula whose weights are positive and sum to $1$. 
 
Define $ \cos \theta_t \coloneqq \frac{\left< -(\nu_t - \mu_K), g_t \right>_K}{\| \mu_K - \nu_t \|_K \| g_t \|_K}$.   We remark on the step size $\gamma_t$; it can be easily confirmed that $\argmin_{\gamma_t \in \zissu} \| \nu_t + \gamma_t g_t - \mu_K \|_K ^2 = \frac{\left< -(\nu_t -\mu_K), g_t \right>_K}{\| g_t \|_K ^2}$ since $\| \nu_t + \gamma_t g_t - \mu_K \|_K ^2$ is a quadratic function of $\gamma_t$ and it holds that
\[\gamma_t = \argmin_{0 \leq \gamma_t \leq 1} \| \nu_t + \gamma_t g_t - \mu_K \|_K ^2 =\min\left\{\frac{\left< -(\nu_t -\mu_K), g_t \right>_K}{\| g_t \|_K ^2} , 1 \right\}\]
 for each $t \in \llbracket 1,   T-1\rrbracket$.

 The following proposition shows that the convergence speed of the worst-case error is influenced by $\cos \theta_t$.

 \begin{proposition}\label{prop2}
Let $A_t \coloneqq \{ i\in   \llbracket 1,   t-1\rrbracket \mid \gamma_i \neq 1\}$. Then, for all $t \in  \llbracket 1,   T\rrbracket$, it is valid that
 \[ \epsilon_{t} \leq \epsilon_1 \prod_{i \in A_t} (1 - \cos^2 \theta_i).  \]
 In particular, if $\gamma_i < 1$ for each $ i\in   \llbracket 1,   t-1\rrbracket $, $\epsilon_{t}  =  \epsilon_1 \prod_{1\leq i \leq t-1} (1 - \cos^2 \theta_i)$.
  \end{proposition}

\begin{proof}
For each $i \in  \llbracket 1,   t-1\rrbracket$, because $0\leq \gamma_i \leq 1$,  the following is valid:
 \begin{equation}\label{eq3-2-1} 
\frac{1}{2}  \| \mu_K - \nu_{i+1} \|_K ^2  =  \frac{1}{2} \| \mu _K- (\nu_i + \gamma_i g_i) \|_K ^2 \leq \frac{1}{2}\| \mu_K - \nu_i  \|_K ^2 .
  \end{equation}
If $\gamma_i \neq 1$, $\gamma_i = \frac{\left< -(\nu_i -\mu_K), g_i \right>_K}{\| g_i \|_K ^2}$. Using this and the definition of $\cos \theta_t$, 
\begin{align}\label{eq3-2-2} 
 \frac{1}{2}\| \mu_K -\nu_{i+1} \|_K  ^2  \notag &= \frac{1}{2}\| \mu_K - (\nu_i  + \gamma_i g_i)\|_K ^2 \notag  \\
&=  \frac{1}{2}\| \mu_K -\nu_{i} \|_K  ^2 + \gamma_i \left< \nu_i - \mu_K, g_i \right>_K + \frac{1}{2} \gamma_i ^2 \| g_i \|_K ^2  \notag \\
&= \frac{1}{2}\| \mu_K -\nu_{i} \|_K ^2- \gamma_i \cos \theta_i  \| \nu_i - \mu_K \|_K \| g_i \|_K +  \frac{1}{2}  \gamma_i ^2  \| g_i \|_K ^2  \notag \\
&=  \frac{1}{2}\| \mu_K  -\nu_{i} \|_K ^2 -  \cos^2 \theta_i   \frac{\| \nu_i - \mu_K \|_K ^2}{2}    \notag \\
&= (1- \cos^2 \theta_i )  \frac{1}{2}\| \mu_K -\nu_{i} \|_K ^2 .
\end{align}
Using (\ref{eq3-2-1}) and (\ref{eq3-2-2}), we obtain the desired result. 
 \end{proof}

In the numerical experiments, which are described later, we had $\gamma_i < 1$ most of the time; thus, the convergence speed of the worst-case error is directly determined by $\cos \theta_i$. We note that if $K_{\max}$ is bounded and $\delta > 0$, the $O(1/t)$ convergence of $\| \mu_K - \nu_t \|_K ^2$ is guaranteed by Proposition~\ref{prop2}. This can be easily confirmed by following the convergence argument of the Frank-Wolfe algorithm for $L$-smooth functions (see, e.g., \cite {jaggi13fw}).
 
 \begin{comment}
Using \autoref{prop2}, we can derive the following corollary, which we proved following the statement in Theorem 3.4 in \cite{pmlr-v119-combettes20a}. 
 \begin{corollary}\label{cor3-1}
Let $\delta > 0$. We assume that each step size $\gamma_t$ in \autoref{boost herding} is determined by line search. Then, if 
\[ | s\in  \llbracket 0, t -1 \rrbracket \mid \gamma_s < 1, K_s > 1 | \geq \omega t^p  \]
 for $\omega > 0$ and $p\in (0, 1)$, for all $t \in \llbracket 1, T \rrbracket $, we have 
\[  \epsilon_t \leq  \epsilon_1   \mathrm{exp}\left( -\delta^2  \omega t^p\right). \]
\end{corollary}

It should be noted that the assumption in \autoref{cor3-1} can be satisfied experimentally for $\omega, p \sim 1$ if we set a small $\delta$. Although the convergence rate in \autoref{cor3-1} is of exponential order,  this does not provide insights regarding the convergence speed of the worst-case error with respect to the number of sample points. This is because the convergence rate in \autoref{cor3-1} is for $t$,  that is, it does not consider the number of iterations $K_t$ for each $t$. In addition, because the exponent of the order is $\delta^2$ and we must set $\delta$ as a small number to satisfy the assumption in \autoref{cor3-1}, this is not fast convergence. 
\end{comment}

To consider the convergence of the worst-case error for a given number of nodes, we use Proposition~\ref{prop2}. It is evident that the convergence speed is directly influenced by $\cos \theta_t$. Therefore, we analyze the behavior of $\cos \theta_t$. First, we state  the convergence speed of positive matching pursuit.

 We prove the following proposition on the convergence speed of positive matching pursuit. To discuss the approximation of \autoref{positive matching pursuit}, we consider the case $\delta = - \infty$, i.e., we do not consider the truncation on line \autoref{compare align} in \autoref{positive matching pursuit}. Proposition~\ref{prop3-3} below states the convergence speed of the positive matching pursuit algorithm at the $t$-th iteration. The procedure of the proof is similar to that in \cite{locatello2017greedy}.  
\begin{proposition}\label{prop3-3}
Let $\delta = -\infty$. If we solve $\min_{d_k \in \mathrm{cone}(V -\nu_t)}  \| -(\nu_t - \mu_K) - d_k \|_K ^2 $ by positive matching pursuit, \[\| -(\nu_t - \mu_K) - d_k \|_K ^2  \leq  \frac{\| \nu_t - \mu_K \|_K ^2}{1+C k \|\nu_t - \mu_K \|_K ^2  }  \]
holds true, where $C$ is a positive constant.
\end{proposition}

For preparation, we use the following lemma. 
\begin{lemma}[\cite{polyak1987introduction, beck2004conditional}]\label{lem3-1}
Let $\{ a_k \}_{k=0}^m$ be a nonnegative sequence of real numbers. If $\{ a_k \}_{k=0}^m$ satisfies $a_{k+1} \leq a_{k} - \gamma a_k ^2\ (k=0, \ldots, m-1)$  for $\gamma > 0$, 
\[ a_m \leq \frac{a_0}{1+m\gamma a_0}. \]
\end{lemma}

Using the above lemma, we prove  Proposition~\ref{prop3-3}. 
 \begin{proof} (Proposition~\ref{prop3-3})
Define $\eta_k = \| -(\nu_t - \mu_K) - d_k \|_k ^2$. First, the following equality holds true:
\begin{equation}\label{eq3-3-1}
  \| -(\nu_t - \mu_K) - d_{k+1} \|_K ^2 = \| -(\nu_t - \mu_K) - d_{k} \|_K ^2 + 2 \left< d_k +\nu_t -\mu_K, d_{k+1} -d_k  \right>_K + \| d_{k+1} -d_k \|_K ^2  .
  \end{equation}
Here, we divide into cases,  as for $u_k \leftarrow \argmax_{u\in \{ v_{k} - \nu_t, -d_k / \| d_k \|_K \}  } \left< r_k, u \right>_K$ in \autoref{positive matching pursuit}: \\
(A) $u_k = v_{k} -\nu_t$\\
The following is evidently valid: $\left< \mu_K -\nu_t - d_k, v_{k} -\nu_t \right>_K \geq \left< \mu_K -\nu_t - d_k,  -d_k / \| d_k \|_K \right>_K  $. In addition, 
\[  \left< \mu_K -\nu_t -d_k, v_{k} -\nu_t \right>_K = \max_{v\in M} \left< \mu_K -\nu_t -d_k, v -\nu_t \right>_K \geq  \left< \mu_K -\nu_t -d_k, \mu_K -\nu_t \right>_K  .\]
Therefore, 
\begin{align*}
\left< \mu_K -\nu_t - d_k, (\| d_k\|_K +1)(v_{k} -\nu_t)\right>_K  &\geq  \left< \mu_K -\nu_t - d_k,  -d_k \right>_K + \left< \mu_K  -\nu_t -d_k, \mu_K  -\nu_t \right>_K \\
&=\| \mu_K -\nu_t -d_k \|_K ^2 \\
&= \eta_k.
\end{align*}
Note that this inequality also holds true for $k=0$. 
By the algorithm, $d_{k+1} = d_k + \frac{\left< \mu_K -\nu_t -d_k, v_{k} -\nu_t \right>_K}{\| v_{k} - \nu_t \|_K ^2} (v_k -\nu_t)$. By substituting this into (\ref{eq3-3-1}), the following holds true:
\[ \eta_{k+1} = \eta_k- \frac{\left< d_k + \nu_t - \mu_K, v_{k}-\nu_t \right>_K ^2 }{\| v_{k} - \nu_t \|_K ^2} \leq \eta_k - \frac{\eta_k ^2}{\| v_{k} - \nu_t \|_K ^2 (\| d_k \|_K +1)^2}. \]\\
(B)\ $u_k =  -d_k / \| d_k \|_K $\\
By the assumption,  it holds that
\begin{align*}
\left< d_k+ \nu_t - \mu_K, -\frac{d_k}{\| d_k \|_K} \right>_K &\leq \left< d_k+ \nu_t - \mu_K, v_{k} -\nu_t \right>_K \\
&= \argmin_{v\in M} \left< d_k+ \nu_t - \mu_K, v -\nu_t \right>_K\\
&\leq \left< d_k + \nu_t -\mu_K, \mu_K -\nu_t \right>_K .
\end{align*}
Therefore, 
\begin{align} \label{prop3-3-2}
 &\left< d_k +\nu_t -\mu_K, -\frac{\| d_k \|_K+1}{\| d_k\|_K} d_k \right> _K \leq \left< d_k +\nu_t -\mu_K, \mu_K -\nu_t - d_k\right>_K = -\eta_k \leq 0\notag  \\
 &\Rightarrow  \left< d_k +\nu_t -\mu_K, -\frac{d_k}{\| d_k\|_K}  \right> _K  ^2 \geq \frac{\eta_k ^2}{(\| d_k \|_K +1)^2} .
 \end{align}
By substituting $d_{k+1}= d_k +  \left< \mu_K -\nu_t -d_k,  -\frac{d_k}{\| d_k\|_K}\right>_K (-\frac{d_k}{\| d_k \|_K})$ into (\ref{eq3-3-1}) and using (\ref{prop3-3-2}), we have
\[ \eta_{k+1} \leq \eta_k -   \frac{\eta_k ^2}{(\| d_k \|_K +1)^2} .\]

From, $(A), (B)$, the following is valid: 
\[ \eta_{k+1} \leq  \eta_{k}  - C {\eta_k ^2 }, \]
where $\sup_{k} (\| d_k \|_K +1)^2 = C_1$, $\max\{ 1, \sup_{v \in M} \|v -\nu_t \|_K ^2 \}=C_2$ and $C_1 C_2 =\frac{1}{C}$. By applying Lemma~\ref{lem3-1} to this, 
\[\eta_{k} \leq \frac{\eta_0}{1+C\eta_0 k} .\]
\end{proof}

\begin{remark}
Note that the constants $C_1$ and $C_2$ in the proof of Proposition~\ref{prop3-3} are bounded regardless of $k$. We first explain  $C_1 =\sup_{k} (\| d_k \|_K +1)^2 $. By the aforementioned proof, $\| d_k - (\mu_K -\nu_t)  \|_K$ decreases monotonically. Therefore, the following holds true:  
 \begin{align*}
 C_1 &\leq  (1+ \| d_k - (\mu_K -\nu_t)  \|_K + \| \mu_K  -\nu_t \|_K)^2   \\
 &\leq   (1+ \| d_0 - (\mu_K -\nu_t)  \|_K + \| \mu_K  -\nu_t \|_K)^2 \\
 &= (1+ 2\| \mu_K  -\nu_t \|_K)^2 . 
 \end{align*}
  The value $\| \mu_K  -\nu_t \|_K ^2$ is bounded because 
\begin{align*}
\| \mu_K  -\nu_t \|_K ^2  &= \int_{\Omega} \int_{\Omega} K(\nu, y) (\mu - \nu_t)(\mathrm{d}\nu)  (\mu - \nu_t)(\mathrm{d}y) \\
&\leq \int_{\Omega}  \left|   \int_{\Omega} K(\nu, y) (\mu - \nu_t)(\mathrm{d}\nu)  \right|  |\mu - \nu_t |(\mathrm{d}y)  \\
&\leq \int_{\Omega}  \|  K  \|_{\infty}  |\mu - \nu_t |(\mathrm{d}y)  \\ 
&\leq   2\| K \|_{\infty}.
\end{align*}
Note that in the above calculation, by abuse of notation, we consider $\nu_t$ as the discrete measure. In the same manner, the boundedness of $C_2$ can be proved. 
\end{remark}

We define $ \cos \theta_{t, k} \coloneqq  \frac{\left< -(\nu_t - \mu_K), d_k \right>_K}{\| \mu_K - \nu_t \|_K \| d_k \|_K}. $
Using Proposition~\ref{prop3-3}, we can derive the lower bound of $\cos \theta_{t, k}$. 

\begin{corollary} \label{cor3-2}
For the same condition as Proposition~\ref{prop3-3}, \[ \cos \theta_{t, k}  \geq  1 -  \frac{1}{2} \cdot \frac{ \frac{ \| \mu_K -\nu_t \|_K}{\| d_k  \|_K}   }{1+Ck \| \mu_K -\nu_t \|_K ^2 } \]
holds true, where $C$ is a positive constant.
\end{corollary}
\begin{proof}
By simple calculation, 
\begin{align*}
\| (\mu_K  -\nu_t) -d_k \|_K ^2 &= \| \mu_K  -\nu_t  \|_K ^2 +\| d_k \|_K ^2 -  2\left< \mu_K  -\nu_t, d_k\right>_K  \\
\Leftrightarrow \cos \theta_{t, k} &=\frac{1}{2}  \frac{\| \mu_K  -\nu_t \|_K}{\| d_k \|} + \frac{1}{2} \frac{\| d_k \|_K}{\| \mu_K -\nu_t \|_K}  -  \frac{\| (\mu_K  -\nu_t) -d_k \|_K ^2}{2\| d_k \|_K \| \mu_K  -\nu_t \|_K} .
 \end{align*}
Using $\frac{1}{2}  \frac{\| \mu_K -\nu_t \|}{\| d_k \|_K} + \frac{1}{2} \frac{\| d_k \|_K}{\| \mu_K -\nu_t \|_K} \geq 1 $, 
\[ \cos \theta_{t, k} \geq 1 -  \frac{\| (\mu_K  -\nu_t) -d_k \|_K ^2}{2\| d_k \|_K \| \mu_K -\nu_t \|_K}. \]
In addition, by applying Proposition~\ref{prop3-3}, 
$ \cos \theta_{t, k}  \geq  1 -  \frac{1}{2} \cdot \frac{ \frac{ \| \mu_K -\nu_t \|_K}{\| d_k  \|_K}   }{1+Ck \| \mu_K -\nu_t \|_K ^2 }.$
\end{proof}

By applying a triangle inequality, we have $\frac{\| \mu_K -\nu_t \|_K}{\| d_k \|_K} \leq  \frac{\sqrt{1+C\epsilon_t k} }{{\sqrt{1+C\epsilon_t k} }-1}$. Using this inequality, we have $\cos\theta_{t, k} \geq 1 -\frac{1}{2} \frac{1}{\sqrt{1+C\epsilon_t k} (\sqrt{1+C\epsilon_t k} -1)}$. From this inequality, it can be observed that as $\epsilon_t$ decreases, the growth of the lower bound of $k$ becomes slower because $k$ is multiplied by $\epsilon_t$. 
Because Corollary~\ref{cor3-2} only yields a lower bound, we cannot conclude that it is difficult to increase $\cos \theta_t$ as $\| \mu_K -  \nu_t  \|_K$ decreases; however, this gives some insight into the convergence speed. As the worst-case error decreases, more points are required to approximate the direction of $-(\nu_t - \mu_K)$. In the experiment, it became difficult to increase $\cos \theta_t$ as $\| \mu_K -  \nu_t  \|_K$ decreased. Considering this property, we propose a method that directly maximizes $\cos \theta_t$.
 
\subsubsection{Greedy maximization of $\cos \theta_t$ method}\label{subs3-2}
Proposition~\ref{prop2} suggests that $\cos \theta_t$ is directly related to the convergence speed of the worst-case error $\| \mu_K - \nu_t \|_K$. Then, if we can maximize $\cos \theta_t$ effectively, we can expect faster convergence. We propose an improved kernel herding algorithm that uses the greedy maximizing method of $\cos \theta_t$ instead of positive matching pursuit. The algorithm is \autoref{greedy cos}. In this algorithm, we greedily add $c ( K(\cdot, y) - \nu_t)$ to the present descent direction $d_k$ to maximize $\cos \theta_{t, k+1}$. We note that Propositions~\ref{prop1} and~\ref{prop2} also hold true if we use \autoref{greedy cos} instead of \autoref{positive matching pursuit} for the approximation of the negative gradient $\mu_K - \nu_t$. For Proposition~\ref{prop1}, because $c_k \geq 0$, it can be ensured that the coefficients of $v_k$ are non-negative. Thus, we can confirm the properties in Proposition~\ref{prop1} by following the same argument. In addition, for the proof of Proposition~\ref{prop2}, we only use the properties of \autoref{boost herding} and the truncation by $\delta$. Therefore, they are also applicable to  \autoref{greedy cos}. We state this in the form of a theorem as follows. 

\begin{theorem}
Propositions~\ref{prop1} and~\ref{prop2} hold true if we apply \autoref{greedy cos} instead of \autoref{positive matching pursuit} to  \autoref{boost herding}. 
\end{theorem}

However, it is not evident that the optimization on line \ref{line1} of \autoref{greedy cos} has a solution with $c >0$. In addition, the optimization procedure on line \ref{line1} is unclear. In the following argument, we discuss these problems. 

Note that we assume $\cos \theta_{t, k}= \frac{\left< -(\nu_t -\mu_K), d_k \right>_K}{\| \mu_K - \nu_t \|_K \| d_k \|_K} < 1$ and $\nu_t \neq \mu_K$ for simplicity.

 We consider an optimization problem 
\[  \argmax_{c \geq 0, v\in V - \nu_t} \frac{\left< -( \nu_t - \mu_K), d_k + c v \right>_K}{\| \mu_K -\nu_t \|_K \| d_k +c v  \|_K} =   \argmax_{c \geq 0, v\in V-\nu_t} \frac{\left< -( \nu_t - \mu_K), d_k + c v \right>_K}{ \| d_k +c v  \|_K}   \]
in \autoref{greedy cos}. For $y\in \Omega$, we denote $\frac{\left< -( \nu_t - \mu_K), d_k + c (K(\cdot, y) - \nu_t) \right>_K}{ \| d_k +c (K(\cdot, y) - \nu_t)   \|_K}  $ by $g(c, y)$. The function $g(c, y)$ can be written as
\[ g(c, y)=\frac{c  \left< -( \nu_t - \mu_K),   K(\cdot, y) - \nu_t \right>_K  + \left< -(\nu_t -\mu_K), d_k \right>_K}{\sqrt{ c^2 \| K(\cdot, y) - \nu_t \|_K ^2 +2c \left< K(\cdot, y) - \nu_t, d_k \right>_K +\| d_k \|_K ^2}}
= \frac{cp + q}{\sqrt{\alpha c^2 + 2 \beta c + \gamma}} ,\] 
where $p \coloneqq \left< -( \nu_t - \mu_K),   K(\cdot, y) - \nu_t\right>_K $, $q \coloneqq \left< -(\nu_t -\mu_K), d_k \right>_K$, $\alpha \coloneqq \| K(\cdot, y) - \nu_t \|_K ^2 $, $\beta \coloneqq \left< K(\cdot, y) - \nu_t, d_k \right>_K $, and $\gamma \coloneqq \| d_k \|_K ^2$. We remark that $p, \alpha, \beta$ are functions of $y$.

Because 
$\frac{\partial g}{\partial c}= \frac{(p\beta -  q\alpha)c - (q\beta - p\gamma)}{(\alpha c^2 + 2 \beta c + \gamma)^{\frac{3}{2}} }, $
for a fixed $y$, $g(c, y)$ takes the extreme value at $c = \frac{q\beta - p\gamma}{p\beta -  q\alpha}$ if $p\beta -  q\alpha \neq 0$.  

We remark that $q > 0$. At each iteration, we maximize $\frac{\left< -( \nu_t - \mu_K), d_k + c v \right>_K}{\| \mu_K -\nu_t \|_K \| d_k +c v  \|_K}$, and thus, $\frac{\left< -( \nu_t - \mu_K), d_k \right>_K}{\| \mu_K -\nu_t \|_K \| d_k   \|_K} \geq \frac{\left< -( \nu_t - \mu_K), d_1 \right>_K}{\| \mu_K - \nu_t \|_K \| d_1  \|_K}$. In addition, because $d_1 =\argmax_{v\in V-\nu_t} \frac{ \left< -(\nu_t - \mu_K),   v \right>_K}{\| v \|_K \| \nu_t - \mu_K \|_K}$ and $\mu_K \in \overline{\mathrm{conv}(V)}$, 
\[ \frac{\left< -( \nu_t - \mu_K), d_1 \right>_K}{\| d_1 \|_K \| \nu_t - \mu_K \|_K} \geq  \frac{\left< -( \nu_t - \mu_K), u_1 \right>_K}{\| u_1 \|_K \| \nu_t - \mu_K \|_K}  \geq \frac{\left< -( \nu_t - \mu_K), \mu_K - \nu_t \right>_K}{\| u_1 \|_K \| \nu_t - \mu_K \|_K}  >  0 ,\]
where $u_1 =\argmax_{v \in V - \nu_t} \left< -( \nu_t - \mu_K), v\right>_K$. Therefore, $q > 0$.

We fix $y\in \Omega$ and discuss the maximum point of $g(c ,y)$ in $c \geq 0$. Proposition~\ref{prop3-4} describes the possibility of the maximum point of $g(c, y)$ for each $y\in \Omega$, and  $g(c, y)$ never takes its maximum at $c=\infty$.

\begin{proposition}\label{prop3-4}
For each $y\in \Omega$, $g(c, y)$ takes its maximum value in $c \geq 0$ at $c=0$ or $c= \frac{q\beta - p\gamma}{p\beta -  q\alpha}$. 
\end{proposition} 
\begin{proof}
First, we consider the case $p\beta -  q\alpha \neq 0$. We divide this into cases $p \geq 0$ and $p < 0$ to clarify the maximum point of $g(c, y)$. \\
(i)\ $p < 0$\\
Because $g(0, y)= \frac{q}{\sqrt{\gamma}} >0$, $\lim_{c \to -\infty} g(c, y)= -\frac{p}{\sqrt{\alpha}} > 0$, and $\lim_{c\to \infty} g(c, y)= \frac{p}{\sqrt{\alpha}}  < 0$, $\argmax_{c\geq 0} g(c, y)$ is $c=0$ or $c=\frac{q\beta - p\gamma}{p\beta -  q\alpha}$. \\
(ii)\ $p \geq 0$\\
Because $g(0, y)= \frac{q}{\sqrt{\gamma}} >0$, $\lim_{c \to -\infty} g(c, y)= -\frac{p}{\sqrt{\alpha}} \leq 0$, and $\lim_{c\to \infty} g(c, y)= \frac{p}{\sqrt{\alpha}}  \geq 0$,  $\argmax_{c\geq 0} g(c, y)$ is $c=0$ or $c=\infty$ or $c=\frac{q\beta - p\gamma}{p\beta -  q\alpha}$. We show that the following never holds true: $\argmax_{c\geq 0} g(c, y) = \infty$. If $\argmax_{c\geq 0} g(c, y) = \infty$, $g(\infty, y) > g(0, y)$, that is, 
\[ \frac{\left< -( \nu_t - \mu_K), K(\cdot, y) - \nu_t \right>_K}{\| \mu_K - \nu_t \|_K \| K(\cdot, y) - \nu_t  \|_K} > \frac{\left< -( \nu_t - \mu_K), d_k \right>_K}{\| \mu_K - \nu_t \|_K \| d_k  \|_K} .\]
In addition, because $\frac{\left< -( \nu_t - \mu_K), d_k \right>_K}{\| \mu_K - \nu_t \|_K \| d_k  \|_K}$  monotonically increases, 
\[   \frac{\left< -( \nu_t - \mu_K), K(\cdot, y) - \nu_t \right>_K}{\| \mu_K - \nu_t \|_K \| K(\cdot, y) - \nu_t  \|_K} > \frac{\left< -( \nu_t - \mu_K), d_1 \right>_K}{\| \mu_K -\nu_t \|_K \| d_1  \|_K} .\]
This is a contradiction because $d_1 = \argmax_{v\in V- \nu_t} \frac{\left< -(\nu_t - \mu_K),   v \right>_K}{\| \mu_K - \nu_t \|\| v\|_K }$. Thus, $\argmax_{c\geq 0} g(c, y)$ is $c=0$ or $c=\frac{q\beta - p\gamma}{p\beta -  q\alpha}$.

If  $p\beta -  q\alpha = 0$, $g(c, y)$ is monotonically decreasing or increasing. Thus, $g(c, y)$  takes its maximum at $c=0$ or $c= \infty$. However, by the above argument, it is evident that $g(c, y)$ never takes its maximum at $c= \infty$. Thus, it takes its maximum at $c= 0$.
\end{proof}

\begin{remark}\label{rem3-2}
We give some  remarks on the proof of Proposition~\ref{prop3-4}.
\begin{enumerate}
\renewcommand{\labelenumi}{(\roman{enumi})}
\item When $p < 0$, if $g(c, y)$ takes the minimum value at $\frac{q\beta - p\gamma}{p\beta -  q\alpha}$, the minimum value is negative. Therefore, if $\frac{q\beta - p\gamma}{p\beta -  q\alpha} > 0$, we can distinguish if the extreme point is the maximum or minimum by the sign.
\item When $p \geq 0$, if $\frac{q\beta - p\gamma}{p\beta -  q\alpha}\geq 0$, this is the maximum point. If $\frac{q\beta - p\gamma}{p\beta -  q\alpha}\geq 0$ and this is the minimum value, $g(c, y) < 0$ for $c < 0$. This is a contradiction because $g(0, y) > 0$.
\end{enumerate}
\end{remark}

Here, we observe \autoref{greedy cos} from another perspective to help with the analysis of the later results. 

Let $d_k$ be the $k$-th approximate direction constructed in \autoref{greedy cos}. We define $P(d)$ for $d \in \mathcal{H}_K$ as the orthogonal projection of $-(\nu_t - \mu_K)$ to the line $\{ \alpha d \mid \alpha \in \zissu \}$. We note that if we replace $d_k$ with $\alpha d_k (\alpha > 0)$ on line \autoref{line1} of \autoref{greedy cos}, $d_{k+1}$ can be written as $\alpha d_{k+1}$. This is obvious because $ \frac{\left< -( \nu_t - \mu_K), d_k + c v \right>_K}{\| \mu_K -\nu _t \|_K \| d_k +c v  \|_K} =  \frac{\left< -( \nu_t - \mu_K), \alpha d_k + \alpha c v \right>_K}{\| \mu_K -\nu _t \|_K \| \alpha d_k + \alpha c v  \|_K}$. Therefore, even if we use $P(d_k)$  instead of $d_k$ in each iteration, the output $g_t$ does not differ from that of the original algorithm. In addition, because $\cos \theta_{t, k} = \frac{\sqrt{\|  -(\nu_t - \mu_K)\|_K ^2 - \|  -(\nu_t - \mu_K)  - P(d_k) \|_K ^2 }}{\|  -(\nu_t - \mu_K)\|_K}$, \autoref{greedy cos} can be interpreted as the greedy minimization of $\|  -(\nu_t - \mu_K)  - P(d_k) \|_K$. Therefore, we can reinterpret the update of \autoref{greedy cos} as follows:
\begin{enumerate}
\item $c_k, v_k \leftarrow  \argmin_{c \geq 0, v \in V- \nu_t}\|  -(\nu_t - \mu_K) -P(d_k + c v) \|_K$.
\item $d_{k+1} =P( d_k + c_k v_k )$.
\end{enumerate}

Now, we consider if there exists $y$ such that the function $g(c, y)$ takes its maximum at $c= \frac{q\beta - p\gamma}{p\beta -  q\alpha} > 0 $. The following Theorem~\ref{thm3-1} suggests a positive result. 
\begin{theorem}\label{thm3-1}
There exist $y\in \Omega$ and $c > 0$ such that $v= K(\cdot, y) - \nu_t $ and $c $ achieve $\argmax_{c \geq 0, v\in V - \nu_t} \frac{\left< -( \nu_t - \mu_K), d_k + c v \right>_K}{\| \mu_K -\nu _t \|_K \| d_k +c v  \|_K}$.
\end{theorem}

\begin{proof}
From Proposition~\ref{prop3-4}, it is sufficient to show that there exists $y \in \Omega$ such that $\frac{\partial g(c, y)}{\partial c}|_{c=0} >  0 $; in this case, $g(c, y)$ does not take its maximum at $c=0$. This can be written as 
\[  p \gamma - q\beta =  \left< -( \nu_t - \mu_K),   K(\cdot, y) - \nu_t\right>_K \| d_k \|_K ^2  -   \left< K(\cdot, y) - \nu_t, d_k \right>_K  \left< -(\nu_t -\mu_K), d_k \right>_K >0 .\]

By the aforementioned argument, we can replace $d_k$ with $P(d_k)$, and it suffices to show that there exists $y\in \Omega$ such that the following inequality holds:
\begin{equation}\label{eq-thm3-1-1}
  \left< -( \nu_t - \mu_K),   K(\cdot, y) - \nu_t\right>_K \| P(d_k) \|_K ^2  -   \left< K(\cdot, y) - \nu_t, P(d_k) \right>_K  \left< -(\nu_t -\mu_K), P(d_k) \right>_K >0 . 
  \end{equation}

Because $P(d_k)$ is orthogonal to $-(\nu_t - \mu_K) -P(d_k)$, it holds that $ \| P(d_k) \|_K ^2  = \left< -(\nu_t -\mu_K), P(d_k) \right>_K$. By substituting this into  (\ref{eq-thm3-1-1}), we only need to show that there exists $y \in \Omega$ such that
\[  \left< -( \nu_t - \mu_K) - P(d_k),   K(\cdot, y) - \nu_t\right>_K  > 0 . \]
 By using the property of orthogonal projection, we have $\|  -( \nu_t - \mu_K)\|_K ^2 > \left< -( \nu_t - \mu_K), P(d_k) \right>_K$ and 
 \[  \left< -( \nu_t - \mu_K) - P(d_k),   -( \nu_t - \mu_K) \right>_K > 0 . \]
 Because $\mu_K \in \overline{\mathrm{conv}(M)}$, we have 
 \[ \max_{y\in \Omega} \left< -( \nu_t - \mu_K) - P(d_k),   K(\cdot, y) - \nu_t\right>_K \geq   \left< -( \nu_t - \mu_K) - P(d_k),  \mu_K - \nu_t \right>_K > 0  \].
 
 This completes the proof. 
\end{proof}

By Theorem~\ref{thm3-1}, we can ensure that $\cos \theta_{t, k}$ monotonically increases, representing the significant difference between this and \autoref{positive matching pursuit}. Because we cannot ensure that $\cos \theta_{t, k}$ monotonically increases in \autoref{positive matching pursuit}, if we set $\delta > 0$, we cannot control when the iteration stops. Therefore, it is difficult to control the total number of nodes. In contrast, in \autoref{greedy cos}, because $\cos \theta_t$ monotonically increases, the maximization of $\cos \theta_t$ is more efficient, and it is easy to estimate the total number of nodes if we set $\delta=0$.

We describe the concrete optimization procedure of $\argmax_{c > 0, y \in \Omega} g(c, y)$. It is difficult to solve $\argmax_{c > 0, y\in \Omega} g(c, y)$ directly using ordinary optimization methods. Therefore,  we prepare a sufficient number of candidate points in $\Omega$. This method is also used in kernel herding to maximize $\left< -(\nu_t - \mu_K) , v - \nu_t \right>_K$. We restrict the candidate points to points that satisfy $\frac{q\beta - p\gamma}{p\beta -  q\alpha}> 0$ and $\frac{q\beta - p\gamma}{p\beta -  q\alpha} p + q \geq 0$. This is due to Proposition~\ref{prop3-4}  and Remark~\ref{rem3-2}.  Then, we select the point that maximizes $g(\frac{q\beta - p\gamma}{p\beta -  q\alpha}, y)$ from the restricted candidate set. The algorithm can be summarized as follows:
\begin{description}
  \item[1.] Prepare a candidate set of sample points.
  \item[2.] Restrict the candidate points to those satisfying  $\frac{q\beta - p\gamma}{p\beta -  q\alpha}> 0$ and $\frac{q\beta - p\gamma}{p\beta -  q\alpha} p + q \geq 0$.
  \item[3.] Select the point $y_k \in \Omega$ that maximizes $g(\frac{q\beta - p\gamma}{p\beta -  q\alpha}, y)$ from the restricted candidate set and let $v_k = K(\cdot, y_k) - \nu_t$ and $c = g(\frac{q\beta - p\gamma}{p\beta -  q\alpha}, y_k)$, where $\alpha, \beta, \gamma, p, q$ are computed for $y_k$. 
\end{description}

\begin{remark}
We note that \autoref{greedy cos} is not computationally expensive. \autoref{greedy cos} requires the computation of $p, q, \alpha, \beta, \gamma$ at each iteration along with the inner products and norms. By memorizing the previous values of $p, q, \alpha, \beta, \gamma$ , we can reduce the cost of such computation; moreover, the computational complexity of \autoref{greedy cos} is not significantly different from that of ordinary kernel herding. Concretely, the computational cost at each iteration is $O(m+t+k)$, where $m$ is the number of candidate points. 
\end{remark}

Next, we analyze the convergence speed of  \autoref{greedy cos}. The following theorem shows that the approximation error between $-(\nu_t - \mu_K)$ and $P(d_k)$, where $d_k$ is constructed in \autoref{greedy cos}, converges with $O(1/k)$ speed. This means the direction of $d_k$ converges to $-(\nu_t - \mu_K)$ and it is sufficient because only the $\cos \theta_i$ influences the  convergence speed. 
\begin{theorem}
Let $d_k$ be the $k$-th approximate direction constructed in \autoref{greedy cos} and let $P(d)$ be the orthogonal projection of $-(\nu_t - \mu_K)$ to the line $\{ \alpha d \mid \alpha \in \zissu \}$. Then, it holds that
\[ \| -(\nu_t - \mu_K) - P(d_k) \|_K ^2 = O(1/k)  \].

\begin{proof}
At first, we note that $P(d_k) = \alpha d_k$ with $\alpha > 0$ because the inner product between $d_k$ and $ -(\nu_t - \mu_K)$ is positive. 

We define $\eta_k \coloneqq \|  -(\nu_t - \mu_K)  - P(d_k) \|_K ^2$. Because we select the optimal vertex $v_k$ and coefficient $c_k$ over $v\in V- \nu_t$ and $c \geq0 $, $\eta_{k} - \eta_{k+1}$ is larger than in any other algorithm that adds a vertex $v$ to $P(d_k)$ with a positive coefficient. Therefore,  by the estimate in the proof of \autoref{prop3-3}, it holds that
\[ \eta_{k+1} \leq  \eta_{k}  - C {\eta_k ^2 }, \]
where $C$ is a positive coefficient. Therefore, we can derive the $O(1/k)$ convergence rate by Lemma~\ref{lem3-1}. 
\end{proof}
\end{theorem}

\subsection{fully-corrective greedy $\cos$ maximization method}\label{FC_gcos_section}

In  \autoref{greedy cos}, we solve the following optimization problem:
\[  \argmax_{c \geq 0, v\in V - \nu_t} \frac{\left< -( \nu_t - \mu_K), d_k + c v \right>_K}{\| \mu_K -\nu_t \|_K \| d_k +c v  \|_K} =   \argmax_{c \geq 0, v\in V-\nu_t} \frac{\left< -( \nu_t - \mu_K), d_k + c v \right>_K}{ \| d_k +c v  \|_K}   . \]
 To obtain a suitably sparse solution, we want to maximize the following quantity:
\[ \frac{ \left< -( \nu_t - \mu_K), d_k  \right>_K}{\| d_k   \|_K \|  -( \nu_t - \mu_K) \|_K} , \]
where $d_k = \sum_{i=1}^k c_k (K(x_i, \cdot) - \mu_K) \coloneqq  \sum_{i=1}^k c_i v_i  \ (c_1, \ldots, c_k \geq 0)$ . 

As we mentioned before, the $\cos$ maximization method can be interpreted as the minimization of $\| -(\nu_t  - \mu_K) - P(d_k)  \|_K$, where $P(d_k)$ is the orthogonal projection of $-(\nu_t - \mu_K)$ to the line $\{ \alpha d_k \mid \alpha \in \zissu \}$. Therefore, the maximization above can be interpreted as the following minimization problem: 
\begin{equation}\label{min-prob1}
 \min_{c_1, \ldots, c_k \geq 0}  \left\|  -(\nu_t - \mu_K) - P\left(\sum_{i=1}^k c_i v_i \right)  \right\|_K  .
 \end{equation}
In addition, the minimization problem (\ref{min-prob1}) is equivalent to the following problem:
\begin{equation}\label{min-prob2}
 \min_{c_1, \ldots, c_k \geq 0}  \left\|  -(\nu_t - \mu_K) - \sum_{i=1}^k c_i v_i  \right\|_K  ^2 .
 \end{equation}
This is because the solution of (\ref{min-prob2}) $d_k ^{*}$ is exactly the orthogonal projection of $-(\nu_t  - \mu_K)$; otherwise, it contradicts the fact that $d_k ^{*}$ is the minimizer. This reformulation is effective for tractability because $\max_{c_1, \ldots, c_k \geq 0} \frac{\left< -(\nu_t - \mu_K), \sum_{i=1}^k c_i v_i \right>_K}{\| \sum_{i=1}^k c_i v_i  \|_K \|  -(\nu_t - \mu_K) \|_K}$ is not easy to solve, but the optimization problem (\ref{min-prob2}) is a constrained convex optimization problem, which can be solved easily.

Therefore, we can construct a variant of \autoref{greedy cos} that optimizes the coefficients $c_1, \ldots, c_k$ in each iteration so that $\cos \theta_{t, k}$ is maximized. This corresponds to the fully-corrective algorithm, and  we can expect improvement of the approximation of the negative gradient.

  We summarize the modified greedy-cos algorithm as follows:

\begin{algorithm}[H]
 \caption{fully-corrective greedy maximization of $\cos \theta_t$}
 \label{fully-corrective}
 \begin{algorithmic}[1]
 \renewcommand{\algorithmicrequire}{\textbf{Input:}  }
 \renewcommand{\algorithmicensure}{\textbf{Output:}} 
\REQUIRE   Input $\nu_t \in M$, maximum number of rounds $K_{\max} \in \sizen$, and truncation parameter $\delta \geq 0$. 

\ENSURE    approximate direction $g_t$
\STATE $d_1 \leftarrow \argmax_{v\in V-\nu_t} \frac{ \left< -(\nu_t - \mu_K),   v \right>_K}{\| v \|_K \| \nu_t - \mu_K \|_K}$
    \FOR {$k=1 $ to $K_{\max}-1$}
  \STATE $c_k, v_k \leftarrow  \argmax_{c \geq 0, v\in V - \nu_t} \frac{\left< -( \nu_t - \mu_K), d_k + c v \right>_K}{\| \mu_K -\nu_t \|_K \| d_k +c v  \|_K}$ 
  \STATE $c_1, \ldots, c_k \leftarrow \argmax_{c_1, \ldots, c_k \geq 0}   \|  -(\nu_t - \mu_K) - \sum_{i=1}^k c_i v_i  \|_K ^2$ \quad(\ref{min-prob2})
    \STATE $d_k ^{'} \leftarrow \sum_{i=1}^k c_i v_i$
  \IF{$\mathrm{align}(-(\nu_t -\mu_K), d_k ^{'}) - \mathrm{align}(-(\nu_t -\mu_K), d_k ) > \delta$ }
 \STATE $d_{k+1} \leftarrow d_k ^{'}$
\ELSE 
\STATE  \textbf{break}
  \ENDIF
  \ENDFOR
  \STATE $K_t \leftarrow k$
    \RETURN $g_t = d_k/ \sum_{i=1}^k c_i$ 
 \end{algorithmic} 
 \end{algorithm}
 
We remark that we can also consider the fully-corrective positive matching pursuit \citep{locatello2017greedy}.
\begin{remark}\label{remark_difference}

One of the advantages of the \autoref{fully-corrective} approach is its computational efficiency. In this algorithm, we only need to solve the optimization problem with at most $K_{\max}$ variables. In contrast, the original fully-corrective kernel herding algorithm requires constrained quadratic programming in each iteration, and the number of variables increases as the number of iterations increases. 
\end{remark}

\subsection{Numerical experiments}\label{Numerical experiments}
We performed numerical experiments to evaluate the performance of the proposed methods. The proposed algorithms are the kernel herding algorithm with positive matching pursuit (\autoref{positive matching pursuit}), that with greedy maximization of $\cos \theta_t$ (\autoref{greedy cos}), and the fully-corrective variants of those algorithms. In the following, we abbreviate these algorithms as ``PMP'', ``gcos'', ``FC-PMP'' and ``FC-gcos'',  respectively. We compare the proposed methods with the existing methods, namely, vanilla kernel herding with step size determined by line search and $\frac{1}{t+1}$ and fully-corrective kernel herding.  For simplicity, we refer to these algorithms as ``linesearch'', ``eq-weight'', and ``FC'', respectively.

\subsubsection*{Gaussian kernel case}
First, we compare the kernel herding  \autoref{positive matching pursuit} and \autoref{greedy cos} to the ordinary kernel herding methods with respect to convergence speed of the worst-case error for a given number of nodes and time. The kernel is a Gaussian kernel $K(x, y)=\mathrm{exp}\left( -\| x - y \|^2  \right)$.  The domain is $\Omega= [-1, 1]^d$, and the density function of the  distribution on $\Omega$ is $\frac{1}{C} \mathrm{exp}\left( -\| x \|^2  \right)$, where $C= \int_{\Omega} \mathrm{exp}\left( -\| x \|^2  \right) \mathrm{d}x$. The experiments were performed for $d=2$. \autoref{gauss_2d} shows the results. It can be observed that the proposed algorithms outperformed line search. Thus, we can confirm that the approximation of the negative gradient $\mu_K - \nu_t$ accelerates convergence. In particular, \autoref{greedy cos} is better than \autoref{positive matching pursuit}. This is reasonable because $\cos \theta_t$ monotonically increases in \autoref{greedy cos}, while it does not in \autoref{positive matching pursuit}.

\begin{figure}[h]
    \begin{minipage}[t]{0.45\hsize}
        \center
        \captionsetup{width=.95\linewidth}
        \includegraphics[keepaspectratio, scale=0.4]{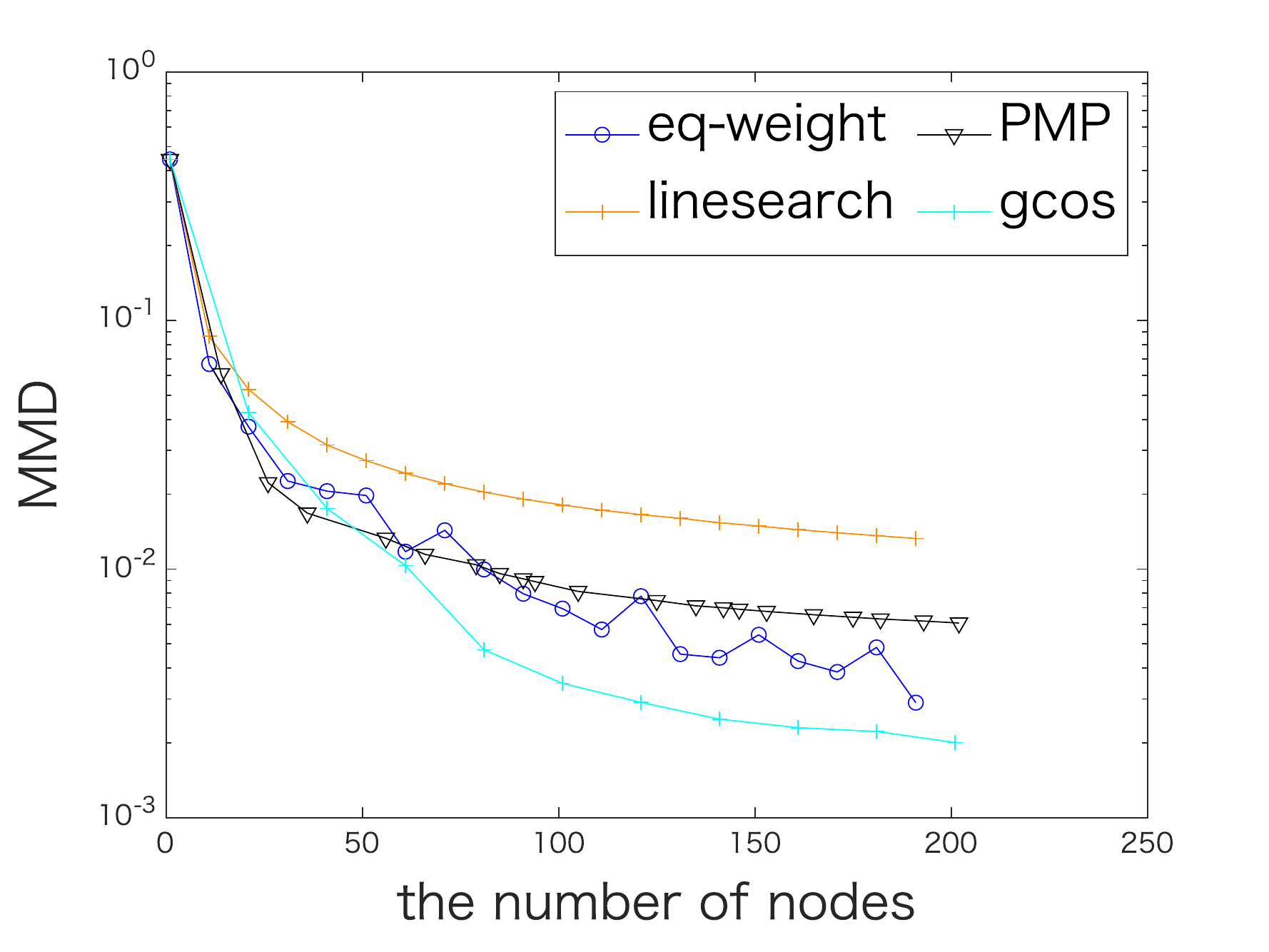}
        \subcaption{MMD for the number of nodes}
        \label{gauss_nodes}
    \end{minipage}
    \begin{minipage}[t]{0.45\hsize}
        \center
        \captionsetup{width=.95\linewidth}
        \includegraphics[keepaspectratio, scale=0.4]{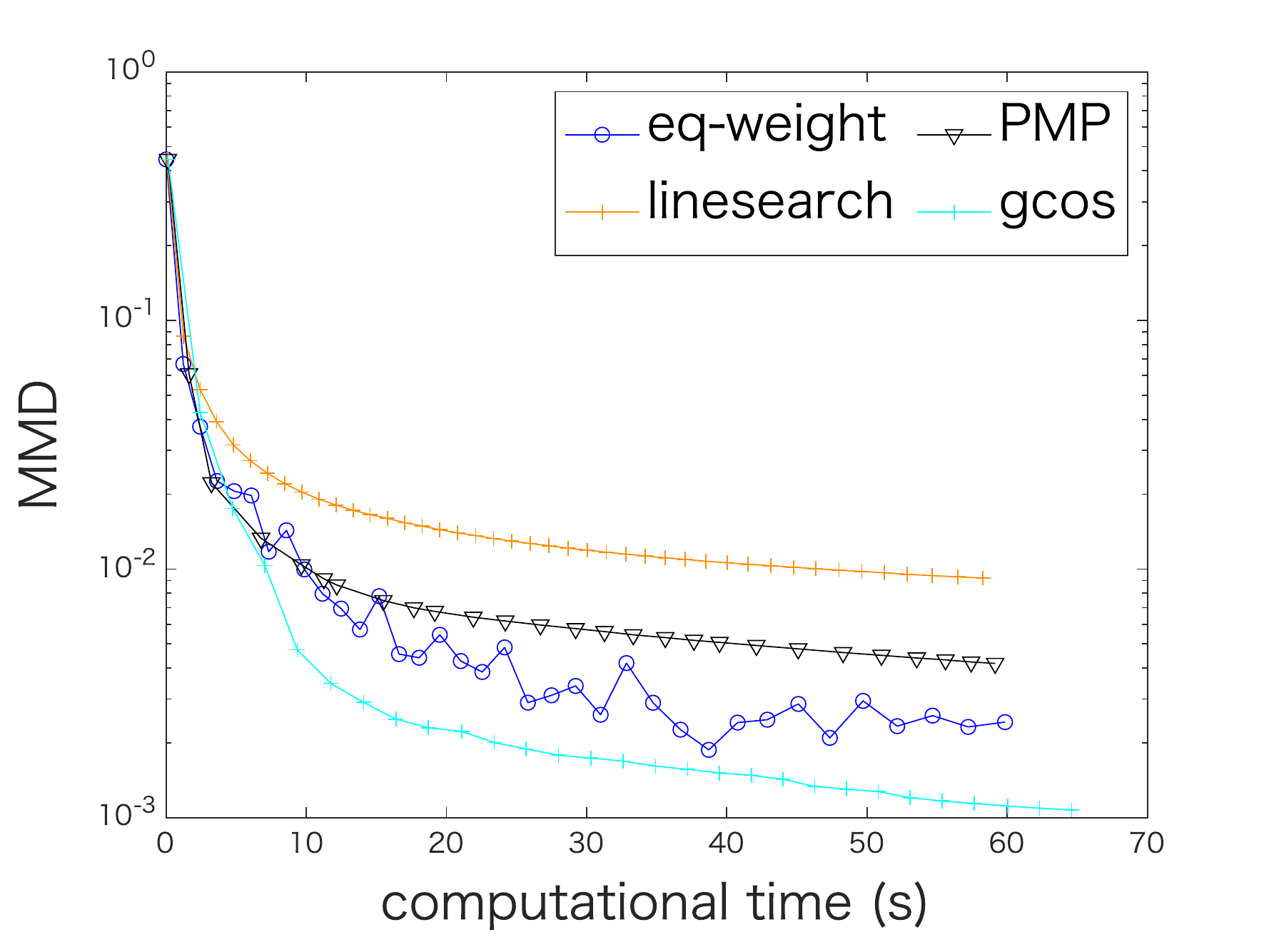}
        \subcaption{MMD for computation time}
        \label{gauss_time}
    \end{minipage}
    \caption{Gaussian kernel ($d=2$)}
    \label{gauss_2d}
\end{figure}
 \autoref{gauss_2d} shows the results. It can be observed that the proposed algorithms outperformed line search. Thus, we can confirm that the approximation of the negative gradient $\mu_K - \nu_t$ accelerates convergence. In particular, \autoref{greedy cos} is better than \autoref{positive matching pursuit}. This is reasonable because $\cos \theta_t$ monotonically increases in \autoref{greedy cos} while it does not in \autoref{positive matching pursuit}.  The proposed algorithms improved the convergence speed for the computation time in comparison to line search. Moreover, ``greedy-cos'' outperformed ``eq-weight''.

\subsubsection*{Mat\'{e}rn kernel case}
We consider the case that the kernel is the Mat\'{e}rn kernel, which has the form
\[ K(x, y)= \frac{2^{1-\nu}}{\Gamma(\nu)} \left( \sqrt{2\nu}\frac{\| x - y \|_2}{\rho} \right)^\nu B_{\nu} \left(\sqrt{2\nu} \frac{ \| x-y \|_2}{\rho} \right), \]
where $B_{\nu}$ is  the modified Bessel function of the second kind, and $\rho$ and $\nu$ are positive parameters. 
The Mat\'{e}rn kernel is closely related to Sobolev spaces and the RKHS $\mathcal{H}_K$ generated by the kernel with parameter $\nu$ norm equivalent to the Sobolev space with smoothness $s=\nu + \frac{d}{2}$ (see, e.g., \cite{kanagawa2018gaussian,wendland2004scattered}). In addition, the optimal convergence rate of the MMD in the Sobolev space with smoothness $s$ is known as $n^{-\frac{s}{d}}$ \citep{novak2006deterministic}. In this section, we use the parameter $(\rho, \nu)=(\sqrt{3}, \frac{3}{2})$ because the kernel has explicit forms with these parameters. 

The domain $\Omega$ is $[-1, 1]^d$, and the probability distribution is uniform. We compare the proposed methods with the existing methods with respect to the convergence of MMD for the number of nodes and computation time for $d=2, 3$. 

The results are shown in \autoref{Matern_gcos}. Regarding the convergence for the number of nodes, although the methods shown in the figures do not achieve the optimal convergence speed, we can see the fully-corrective approach outperforms the original methods. Moreover, the fully-corrective variants also achieve fast convergence speeds for the computation time. 

\begin{figure}[h]
    \begin{minipage}[t]{0.45\hsize}
        \center
        \captionsetup{width=.95\linewidth}
        \includegraphics[keepaspectratio, scale=0.4]{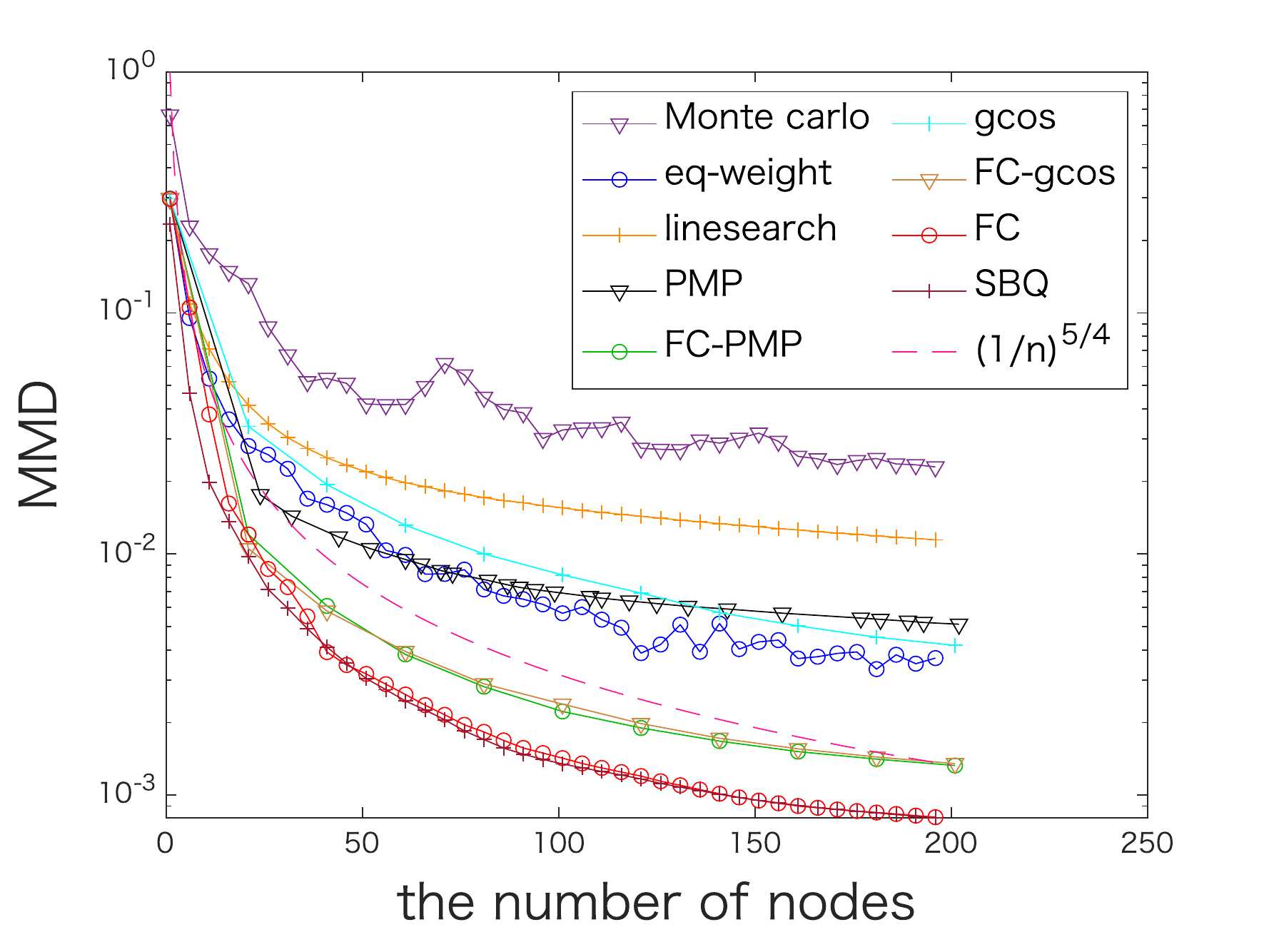}
        \subcaption{MMD for the number of nodes ($d=2$)}
        \label{matern_nodes_2d}
    \end{minipage}
    \begin{minipage}[t]{0.45\hsize}
        \center
        \captionsetup{width=.95\linewidth}
        \includegraphics[keepaspectratio, scale=0.4]{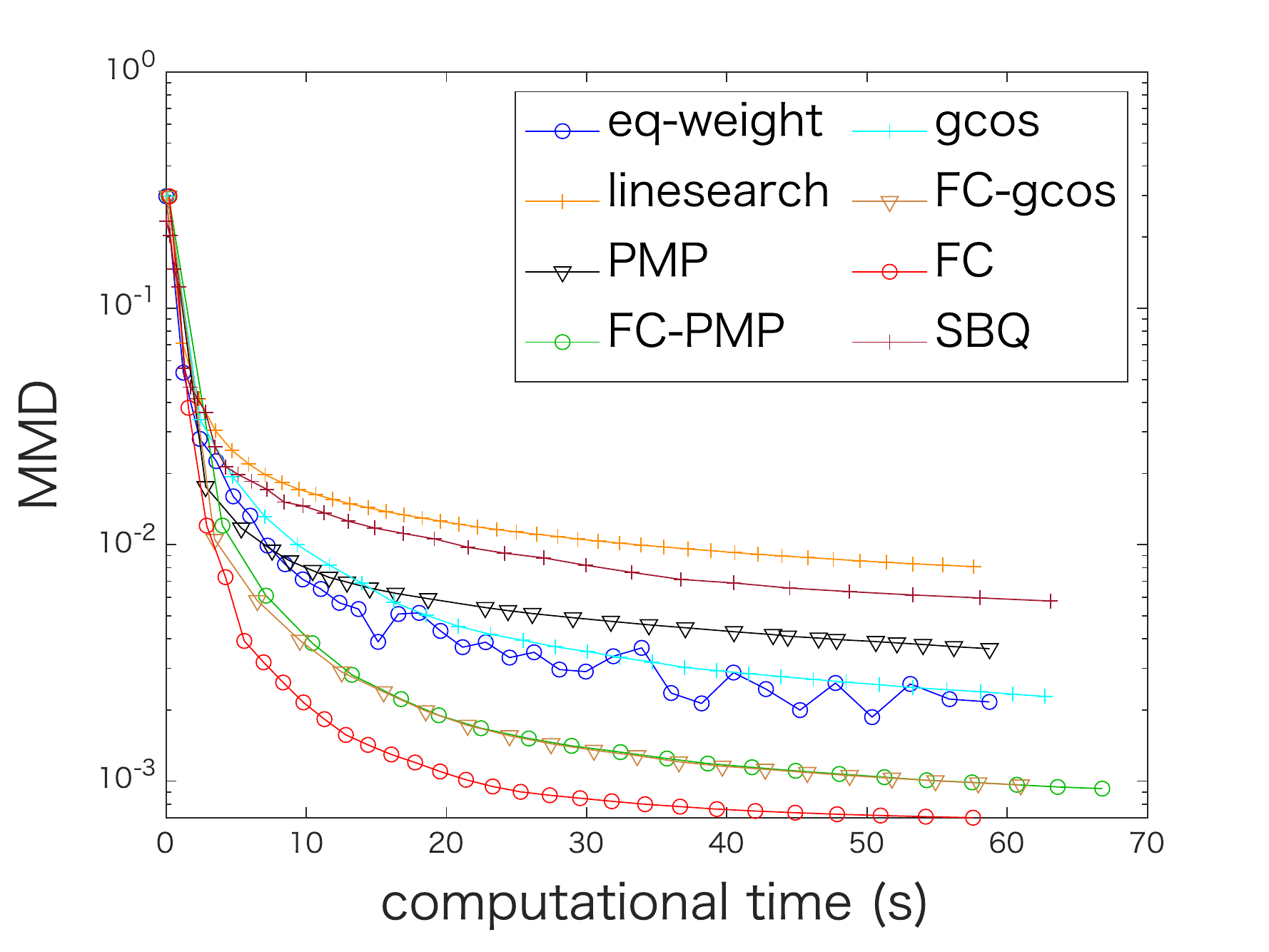}
        \subcaption{MMD for computation time ($d=2$)}
        \label{matern_time_2d}
    \end{minipage}\\
        \begin{minipage}[t]{0.45\hsize}
        \center
        \captionsetup{width=.95\linewidth}
        \includegraphics[keepaspectratio, scale=0.4]{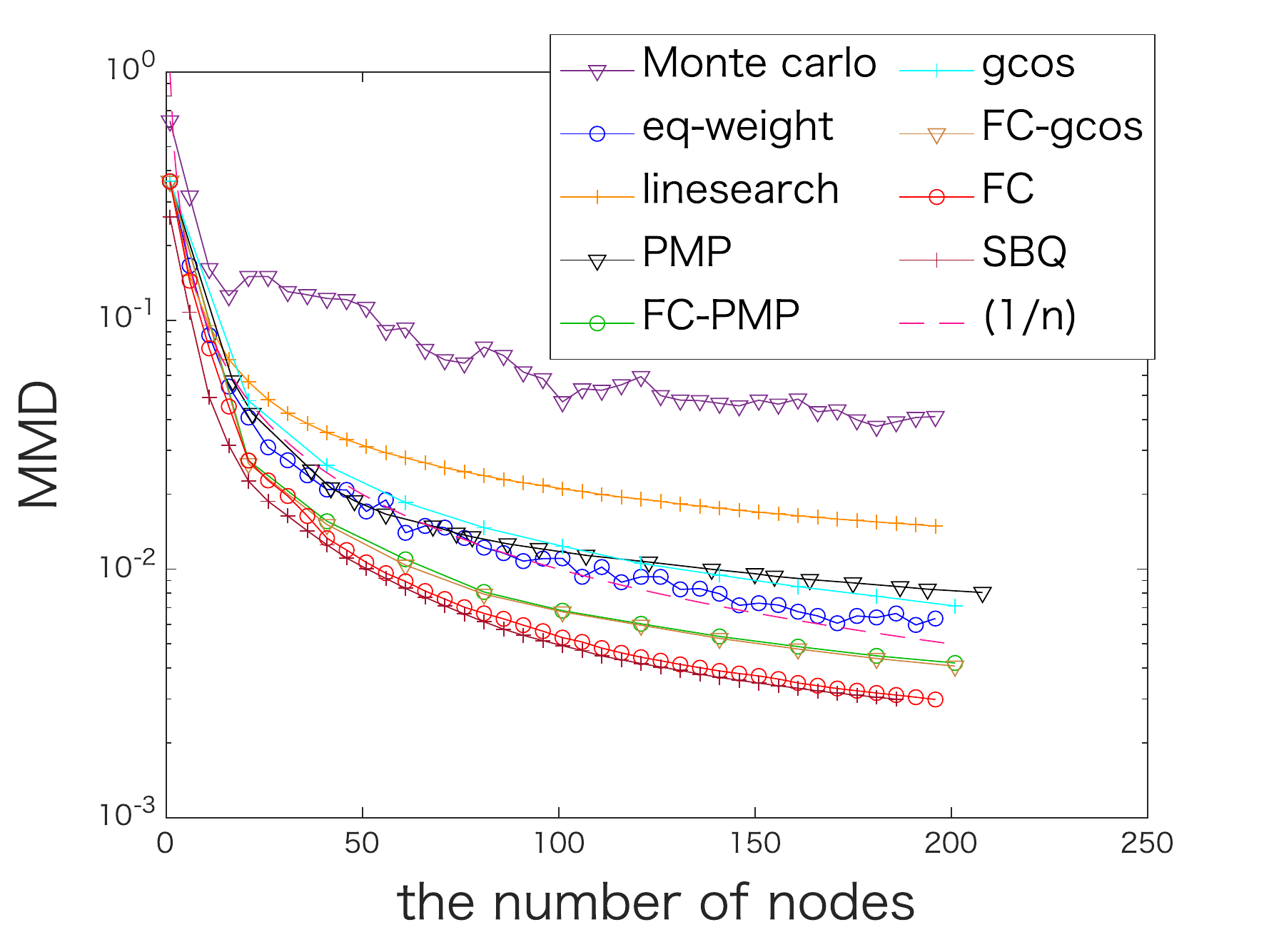}
        \subcaption{MMD for the number of nodes ($d=3$)}
        \label{matern_nodes_3d}
    \end{minipage}
    \begin{minipage}[t]{0.45\hsize}
        \center
        \captionsetup{width=.95\linewidth}
        \includegraphics[keepaspectratio, scale=0.4]{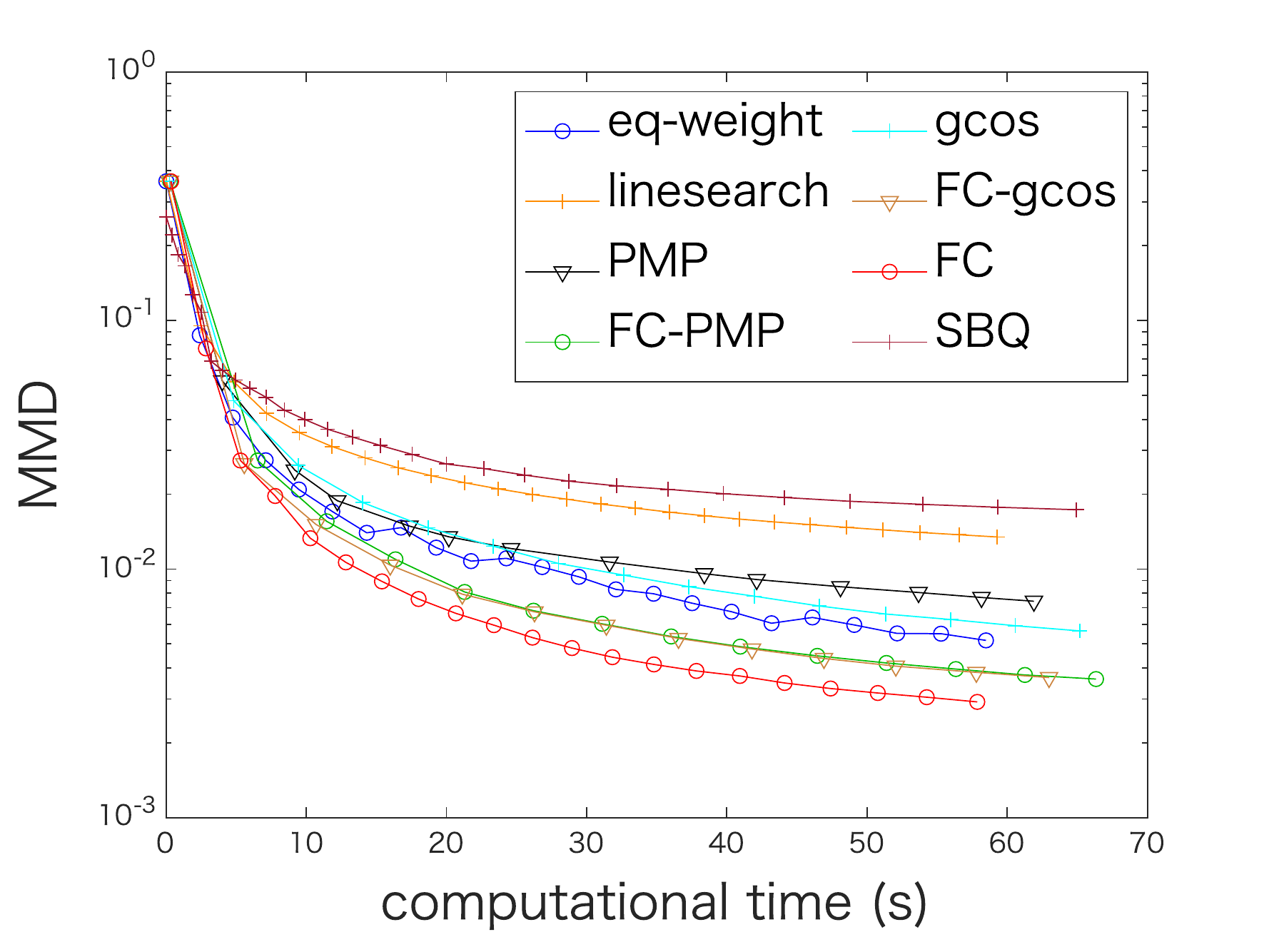}
        \subcaption{MMD for computation time ($d=3$)}
        \label{matern_time_3d}
    \end{minipage}
    \caption{Mat\'{e}rn kernel ($\nu=\frac{3}{2}$)}
    \label{Matern_gcos}
\end{figure}

\subsubsection*{Kernel herding on a sphere}
We also performed an experiment for integration on a sphere in $\zissu^3$. The kernel is $K(x, y)= \frac{8}{3} - \| x -y \|$. The domain $\Omega$ is a unit sphere centered at the origin in $\zissu^3$, and the probability distribution is uniform. It is known that the optimal rate of the worst-case error is $1/n^{\frac{3}{4}}$,  as demonstrated by \cite{brauchart2014qmc}.  In this setting, we compared the fully-corrective versions of the proposed algorithms with the ordinary fully-corrective kernel herding. \autoref{Sphere_graph} shows the results. We can see in \autoref{Sphere_graph} that all three algorithms achieve the optimal convergence speed for the number of nodes. Regarding computation time, ``FC-PMP'' and ``FC-gcos'' outperformed the ordinary fully-corrective kernel herding. As we mentioned in \autoref{FC_gcos_section}, the fully-corrective variants of both proposed algorithms are computationally efficient compared to the ordinary fully-corrective variant as the number of nodes increases. This observation explains the outperformance in terms of computation time.

\begin{figure}[h]
    \begin{minipage}[t]{0.45\hsize}
        \center
        \captionsetup{width=.95\linewidth}
        \includegraphics[keepaspectratio, scale=0.4]{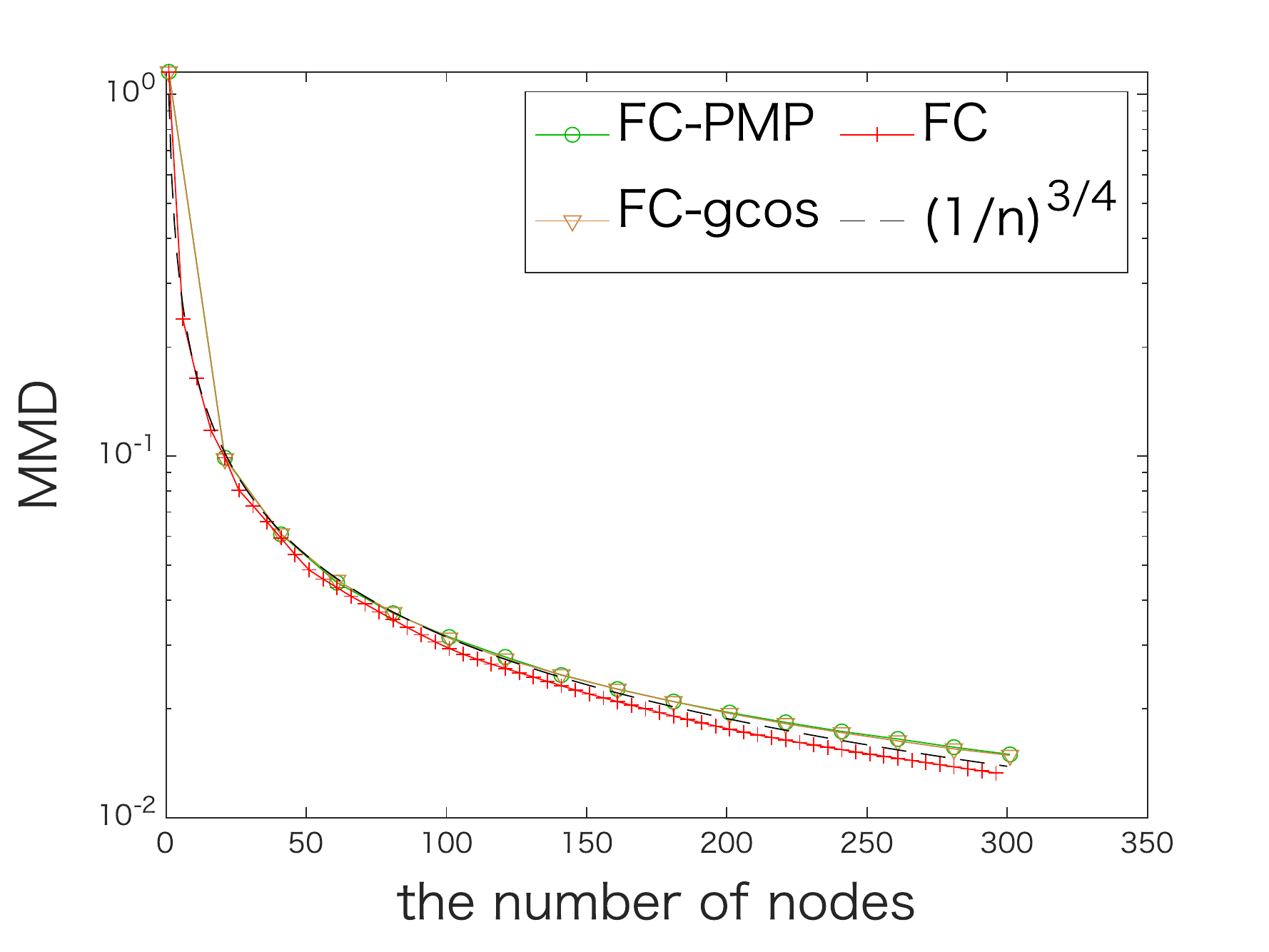}
        \subcaption{MMD for the number of nodes}
        \label{gauss_nodes}
    \end{minipage}
    \begin{minipage}[t]{0.45\hsize}
        \center
        \captionsetup{width=.95\linewidth}
        \includegraphics[keepaspectratio, scale=0.4]{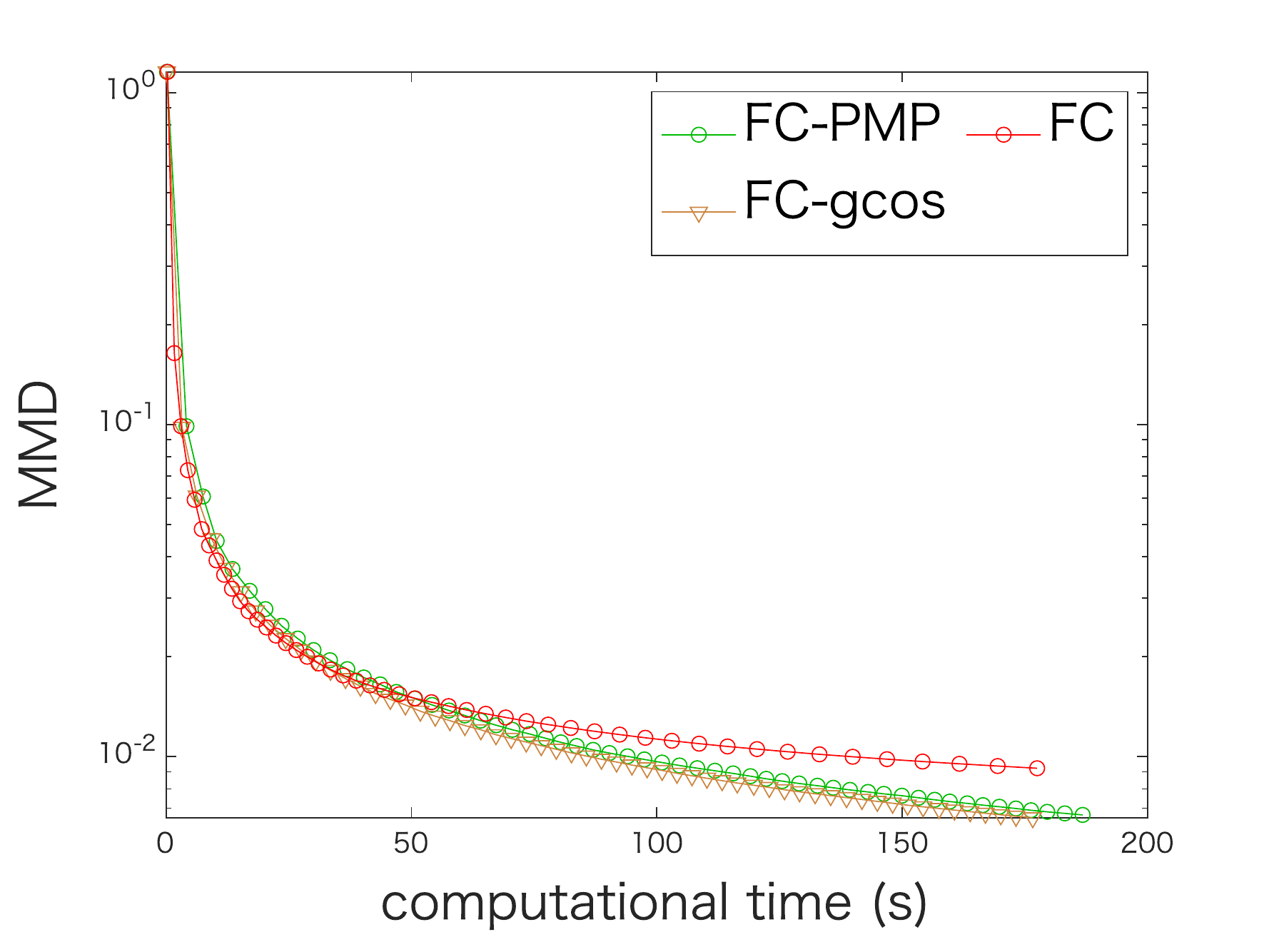}
        \subcaption{MMD for computation time}
        \label{gauss_time}
    \end{minipage}
    \caption{Kernel quadrature on a unit sphere}
    \label{Sphere_graph}
\end{figure}
 
%%Coment Out%%
\begin{comment}
%\subsection{Summary of this section}
In this section, we improve the kernel herding algorithm to achieve faster convergence speed while preserving the numerical stability and computational efficiency. The main concept involves approximating the negative gradient $-\nabla_{\nu} F(\nu_t) = \mu_K - \nu_t$. \autoref{positive matching pursuit} implemented the algorithm of \cite{pmlr-v119-combettes20a}. We confirmed that the results in \cite{pmlr-v119-combettes20a} are also valid for kernel herding and analyzed the convergence speed of the approximation of the negative gradient $\mu_K -\nu_t$. \autoref{greedy cos} involved the maximization of $\cos \theta_t$, where $\theta_t$ denotes the angle between $\mu_K - \nu_t$ and the approximate direction. We  proved the theoretical guarantee for the monotonic increase of $\cos \theta_{t, k}$ and convergence of the approximate gradient to the true negative gradient. In addition, we propose the fully-corrective versions of these algorithms. Through numerical experiments, we confirmed that the proposed algorithms improve the convergence speed for a given number of nodes and computation time; moreover, the improvement in \autoref{greedy cos} is remarkable, and the fully-corrective variants achieved convergence speeds competitive with the optimal convergence rates.
\end{comment}

\section{Theoretical analysis of fully-corrective kernel quadrature rules}\label{chapter5}

As shown in the numerical experiments in \autoref{Numerical experiments}, the fully-corrective kernel herding algorithm and proposed algorithm \autoref{boost herding} with  fully-corrective gradient approximation shown in \autoref{FC_gcos_section} perform significantly well in terms of convergence speed for the number of nodes. 
To observe the results theoretically, we analyze the convergence speed of the kernel quadrature formula with the weights $\{\omega_1, \ldots \omega_n\}$ optimized over the convex hull $\{ \omega_1, \ldots, \omega_n \mid \sum_{i=1}^n \omega_i =1, \omega_i \geq 0 \ (i=1, \ldots, n) \}.$ This problem setting is closely related to that of fully-corrective kernel herding (e.g., \cite{holloway1974extension, jaggi13fw}), which executes the optimization of the weights over the convex hull in each iteration. In addition, it is also closely related to the fully-corrective variants of \autoref{boost herding} in \autoref{chapter3}. Although the significant practical performance was confirmed in previous studies, such as \cite{10.5555/3042573.3042747, pmlr-v38-lacoste-julien15} and the previous section, the theoretical performance of the optimized weights over the convex hull was not analyzed sufficiently. In detail, although the square root convergence speed of the worst-case error was confirmed for any kernel function, theoretical analysis considering the properties of the kernel functions, such as smoothness, has not yet been conducted.

\subsection{Convergence analysis of fully-corrective kernel quadrature rules}
In this section, we analyze the theoretical aspect of the kernel quadrature rules with fully-corrective weights. This analysis is closely related to fully-corrective kernel herding and the fully-corrective variants of the algorithms introduced in \autoref{FC_gcos_section}. Although the set of nodes is fixed in the following theorem and algorithm-dependent analysis has not yet been conducted, we consider this theorem to help the analysis of the fully-corrective algorithms. 

\begin{theorem}\label{thm-FC-2}

We assume that constant functions are contained in $\mathcal{H}_K$. For the nodes $X=\{ x_1, \ldots, x_k \} \subset \Omega$, if it holds that 
\[ \| f - s_{f, X}  \|_{\infty}  \leq  \| f \|_K  r_k ,\]
 there exists a constant $C> 0$ and we have 
 \[\max_{x\in \Omega}  \left| \left< \mu_K  - \sum_{i=1}^k \omega_i K(x_i , \cdot),  K(x, \cdot) - \sum_{i=1}^k \omega_i K(x_i , \cdot) \right> \right|\leq C  r_k ,\]
where $\omega=(\omega_1, \ldots, \omega_k) = \argmin_{\omega \in \Delta(k)} \| \mu_K - \sum_{i=1}^k \omega_i K(x_i, \cdot) \|_K, \Delta(k)=\{ \omega_1, \ldots, \omega_k \mid \sum_{i=1}^k \omega_i =1, \omega_i \geq 0 \ (i=1, \ldots, k) \}$. In addition to this, we assume $\omega_i > 0\ (i=1, \ldots, k)$.  Moreover, it holds that
\[  \left\| \mu_K -  \sum_{i=1}^k \omega_i K(x_i , \cdot) \right\|_K \leq \sqrt{C  r_k}.  \]
\end{theorem}
 \begin{proof}
We show that for each $j \in \{ 1, \ldots, k\}$, it holds that
\[ \left<\mu_K -\sum_{i=1}^k \omega_i K(x_i, \cdot), K(x_j, \cdot) - \sum_{i=1}^k \omega_i K(x_i, \cdot) \right>_K = 0 .\]
 Let $j $ be an integer in $\{ 1, \ldots, k\}$. If $\left<\mu_K -\sum_{i=1}^k \omega_i K(x_i, \cdot), K(x_j, \cdot) - \sum_{i=1}^k \omega_i K(x_i, \cdot) \right>_K  > 0$, the  following quadratic function with respect to $\alpha $
\[ \left\| \mu_K - \sum_{i=1}^k \omega_i K(x_i, \cdot) - \alpha \left( K(x_j, \cdot) -   \sum_{i=1}^k \omega_i K(x_i, \cdot) \right)  \right\|_K ^2  \]
takes its minimum at $\alpha = \frac{\left<  \mu_K - \sum_{i=1}^k \omega_i K(x_i, \cdot),  K(x_j, \cdot) -   \sum_{i=1}^k \omega_i K(x_i, \cdot) \right>_K}{\| K(x_j, \cdot) -   \sum_{i=1}^k \omega_i K(x_i, \cdot)\|_K ^2} > 0$. This contradicts the minimality of $\left\| \mu_K -  \sum_{i=1}^k \omega_i K(x_i , \cdot) \right\|_K ^2$ because we can decrease the  function value by replacing $\sum_{i=1}^k \omega_i K(x_i, \cdot)$ with 
\[(1-\alpha)\sum_{i=1}^k \omega_i K(x_i, \cdot) + \alpha K(x_j, \cdot)\]
 for a sufficiently small $\alpha > 0$. Next, we consider the case of 
 \[\left<\mu_K -\sum_{i=1}^k \omega_i K(x_i, \cdot), K(x_j, \cdot) - \sum_{i=1}^k \omega_i K(x_i, \cdot) \right>_K  <  0.\]
  By applying the same argument to the function
\[ \left\| \mu_K - \sum_{i=1}^k \omega_i K(x_i, \cdot) - \alpha \left( \sum_{i=1}^k \omega_i K(x_i, \cdot)  - K(x_j, \cdot ) \right) \right\|_K ^2 , \]
we can derive a contradiction. Therefore, it holds that
 
 \begin{equation}\label{eq-optimized1}
  \left<\mu_K -\sum_{i=1}^k \omega_i K(x_i, \cdot), K(x_j, \cdot) - \sum_{i=1}^k \omega_i K(x_i, \cdot), \right>_K =0\quad (j=1, \ldots, k).
\end{equation}

We set $z= (\mu_K(x_1), \ldots, \mu_K(x_k)) ^\top$, $\epsilon(X)= \left< \sum_{i=1}^k \omega_i K(x_i, \cdot), \mu_K -\sum_{i=1}^k \omega_i K(x_i, \cdot) \right>_K$, and  $K= (K(x_i, x_j))_{i, j}$.   Then, (\ref{eq-optimized1}) can be rewritten as
\begin{equation}\label{eq-optimized2} 
 K \omega = z - \epsilon(X) \mathbbm{1} \quad (\mathbbm{1} = (1, \ldots, 1)^\top).
 \end{equation}

 By the equality (\ref{eq-optimized2}), the weights $\omega_1, \ldots, \omega_k$ can be considered as the weights of the interpolation of $\mu_K - \epsilon(X) \bm{1}$, where $\bm{1}$ is a constant function that identically outputs $1$.  By the Cauchy-Shwarz inequality, we can bound $\epsilon(X)$ as follows:
 \[ \epsilon(X)= \left<\mu_K -\sum_{i=1}^k \omega_i K(x_i , \cdot), \sum_{i=1}^k \omega_i K(x_i, \cdot) \right>_K \leq  2 \sup_{\| f \|_K \leq 1} \sup_{x\in \Omega} |f(x)| \leq  2\| K \|_{\infty}  .\]
 We note that $\|\mu_K \|_K$ and $ \|\sum_{i=1}^k \omega_i K(x_i, \cdot)\|_K $ are bounded by $ \sup_{\| f \|_K \leq 1} \sup_{x\in \Omega} |f(x)|$. Therefore, by using the assumption for the constant function, we can bound the RKHS norm of  $\mu_K - \epsilon(X) \bm{1}$ by a positive constant. Therefore, we can use the assumption, and it holds that
 \begin{align*}
 &\max_{x\in \Omega}  | \mu_K (x) -  \epsilon(X) \bm{1}(x) - \sum_{i=1}^k \omega_i K(x_i , x)  |\leq C^{'} r_k  \\
 &\Leftrightarrow   \max_{x\in \Omega}  \left| \left< \mu_K  - \sum_{i=1}^k \omega_i K(x_i , \cdot),  K(x, \cdot) - \sum_{i=1}^k \omega_i K(x_i , \cdot) \right>_K \right|\leq C^{'}  r_k  .
 \end{align*}
 In addition, using $\mu_K \in M$, we have
 \begin{align*}
  \| \mu_K  - \sum_{i=1}^k \omega_i K(x_i , \cdot)\|_K ^2   &= \left<   \mu_K  - \sum_{i=1}^k \omega_i K(x_i , \cdot),  \mu_K  - \sum_{i=1}^k \omega_i K(x_i , \cdot) \right>_K  \\
  &\leq \max_{x\in \Omega}\left< \mu_K  - \sum_{i=1}^k \omega_i K(x_i , \cdot),  K(x, \cdot) - \sum_{i=1}^k \omega_i K(x_i , \cdot) \right>_K .
\end{align*}
 Thus, we have the desired inequality. 
 \end{proof}
 
 \begin{remark}
 Theorem~\ref{thm-FC-2} claims that the  distribution of nodes for the kernel interpolation is also effective for the kernel quadrature rule. For example, if the domain $\Omega$ satisfies some conditions, the kernel interpolation in the Sobolev space  $W_2 ^\beta (\Omega) \ (\beta > d/2)$ 
 satisfies 
 \[ \| f - s_{f, X} \|_\infty \leq C \| f \|_{W_2 ^\beta (\Omega)} h_{X,\Omega} ^{\beta - \frac{d}{2}} ,\]
 where $f \in W_2 ^\beta (\Omega), X= \{ x_1, \ldots, x_n \} \subset \Omega $ and $h_{X,\Omega} = \sup_{x\in \Omega} \min_{x_j \in X} \|  x - x_j\|_2$. In this regard, we refer the reader to \cite{wu1993local, schaback1995error}. Therefore, if $h_{X,\Omega} \leq c_1 n^{-\frac{1}{d}}$, the worst-case error is upper bounded by $n^{-\frac{\beta}{2d} +\frac{1}{4}}$. If $\beta$ is sufficiently large, the convergence rate is faster than the known rate $n^{-\frac{1}{2}}$.  
 
 However, the upper bound in  Theorem~\ref{thm-FC-2}  is not optimal. It is known that the optimal rate of the error of the interpolation and quadrature is $n^{-\frac{\beta}{d}}$  for the Sobolev space  $W_2 ^\beta (\Omega)$ \citep{jerome1970n, novak2006deterministic}. Therefore, Theorem~\ref{thm-FC-2} can only  show the convergence rate up to $n^{-\frac{\beta}{2d}}$, but it is not optimal. Thus, there is room for improvement for the theoretical analysis of the fully-corrective kernel quadrature. 
  \end{remark}

\begin{comment}
%\subsection{Conclusion and future work}
In this chapter, we provided new theoretical analysis of quadrature rules with fully-corrective weights. The kernel-specific approach gives a new perspective to fully-corrective methods. However, there is much work that has not yet been conducted. We consider that one of the goals of this research is deriving the convergence speeds for the number of nodes for the fully-corrective algorithms, such as the fully-corrective kernel herding and proposed methods in this paper. Therefore, we should  further conduct new analysis of the fully-corrective method. 
\end{comment}

\section{Conclusion}
In this paper, to derive quadrature rules with sparser nodes by kernel herding, we proposed the improved kernel herding algorithm whose concept is approximating the negative gradient by several vertex directions. We proposed the two gradient approximation methods \autoref{positive matching pursuit} and \autoref{greedy cos} and their fully-corrective versions. We provided theoretical analysis of the algorithms, and numerical experiments showed significant improvements in convergence speed of the integration error for the sparsity of nodes and computation time.  In \autoref{chapter5}, we studied the convergence properties of the fully-corrective kernel quadrature formulas. We provided a new analysis using the relationship with kernel interpolation and showed a better convergence rate than the known rate $O(1/\sqrt{n})$. This new analysis gave partial theoretical support to the remarkable performance of the fully-corrective kernel herding and fully-corrective variants of \autoref{boost herding} proposed in \autoref{FC_gcos_section}. 

\subsection*{Future work}
\begin{itemize}
\item Although we confirmed the practical performance of the proposed algorithms in terms of solution sparsity,  theoretical analysis on the sparsity has not been sufficiently conducted. Therefore, convergence analysis of the proposed algorithms for the sparsity of nodes must be further explored. The results in \autoref{chapter5} may help us to analyze the sparsity.  
\item 
Application of the algorithms in \autoref{chapter3} to finite-dimensional problems might be effective, especially for the sparsity of solutions. We can test the application in future work.

\end{itemize}

\acks{This work was partly supported by JST, PRESTO Grant Number JPMJPR2023, Japan. We would like to thank Editage (www.editage.com) for English language editing.}


\vskip 0.2in
\bibliography{refs}

\newpage

\appendix
\section{Proofs}\label{Proofs}

 \subsection{Proof of $\mu_K \in M$}
   \begin{lemma}\label{proof of lemma}
   If $K(\cdot, \cdot)$ is uniformly continuous on $\Omega \times \Omega$ and bounded, for any Borel probability measure $\mu$ and $\mu_K$, which is the embedding of $\mu$, there exists $\{ x_i \}_{i=1}^\infty \subset \Omega$ such that 
   \[ \mu_K \in \overline{\mathrm{conv}( \{ K(x_i, \cdot) \}_{i=1}^\infty} ),\]
   where closure is considered with respect to $\| \cdot \|_{K}$.
   \end{lemma}
 \begin{proof}
 Let $\{X_i\}_{i=1}^\infty$ be an i.i.d sequence of r.v.  that satisfies $X_i \sim \mu\ (i=1, 2, \ldots)$. By the law of large numbers, for any $y \in \Omega$, the following holds true: 
 \begin{equation} \label{eq-lem1}
  \lim_{n\to \infty} \sum_{i=1}^n \frac{1}{n} K(X_i, y) \to  \int_{\Omega} K(x, y) \mu(\mathrm{d}x) 
 \end{equation}
 with a probability of 1. Because $\zissu^d$ is separable, we can take $\{ y_i \}_{i=1}^\infty \subset \Omega$, which is dense in $\Omega$. In addition, because an intersection of the countable sets of measure $1$ is  measure $1$, from (\ref{eq-lem1}), there exists $\{ x_i \}_{i=1}^\infty \subset \Omega$ such that
 \begin{equation} \label{eq-lem2}
  \lim_{n\to \infty} \sum_{i=1}^n \frac{1}{n} K(x_i, y_j) \to  \int_{\Omega} K(x, y_j) \mu(\mathrm{d}x) 
 \end{equation}
for any $y_j \in \{ y_i \}_{i=1}^\infty$. By the assumption,  $K$ is uniformly continuous in $\Omega \times \Omega$. Thus, for any $\epsilon > 0$, there exists $\delta > 0$ such that if $| (x_1, y_1) - (x_2, y_2) | < \delta$, then $| K(x_1, y_1) - K(x_2, y_2) | < \epsilon$. In addition, for any $y\in \Omega$ and $\delta > 0$, there exists $y_j \in \{ y_i \}_{i=1}^\infty $ such that $| y - y_j | < \delta$. Therefore, for any $\epsilon > 0$ and $y\in \Omega$, we take $y_j$ such that $|y -y_j| <\delta$ and the following is valid: 
\[ \left| \sum_{i=1}^n \frac{1}{n} K(x_i, y_j) - \sum_{i=1}^n \frac{1}{n} K(x_i, y)  -  \left( \int_{\Omega} K(x, y_j) \mu(\mathrm{d}x)  - \int_{\Omega} K(x, y) \mu(\mathrm{d}x)  \right) \right| < 2\epsilon .  \] 
By (\ref{eq-lem2}), we let $n \to \infty$ and 
\[ \lim_{n \to \infty} \left|  \sum_{i=1}^n \frac{1}{n} K(x_i, y_j) -  \int_{\Omega} K(x, y_j) \mu(\mathrm{d}x)    \right| < 2\epsilon .\]
for any $\epsilon > 0$. This means that for each $y\in \Omega$, 
\[  \lim_{n\to \infty}\left<  \sum_{i=1}^n \frac{1}{n} K(x_i, \cdot), K(y, \cdot)\right>_K \to  \left< K(y, \cdot), \mu_K \right>_K.\]
Because the subspace of $\mathcal{H}_K$ spanned by $\{ K(y, \cdot) \mid y\in \Omega \}$ is dense in $\mathcal{H}_K$, $\mu_K \in \overline{\mathrm{conv}( \{ K(x_i, \cdot) \}_{i=1}^\infty} ) $.
 \end{proof}

\end{document}